\documentclass[12pt, a4paper]{amsart}

\usepackage[hmargin=30mm, vmargin=25mm, includefoot, twoside]{geometry}
\usepackage[bookmarksopen=true]{hyperref}
\usepackage[usenames,dvipsnames]{xcolor}

\usepackage{amscd}
\usepackage{amsfonts,amssymb,verbatim}
\usepackage{latexsym}
\usepackage{mathrsfs}
\usepackage{stmaryrd}
\usepackage{xspace}
\usepackage{enumerate, paralist}
\usepackage{graphicx}
\usepackage[all]{xy}
\usepackage{extarrows}
\usepackage{tikz-cd}

\usepackage{txfonts, pxfonts}

\usepackage{amsthm}
\usepackage{amsmath}

\usepackage{todonotes}
\newtheorem{thm}{Theorem}[section]
 \newtheorem{cor}[thm]{Corollary}
 \newtheorem{lem}[thm]{Lemma}
 \newtheorem{prop}[thm]{Proposition}
 
 \newtheorem{Q}[thm]{Question}

\numberwithin{equation}{section}

 \theoremstyle{definition}
  \newtheorem{defn}[thm]{Definition}

 \theoremstyle{remark}
 \newtheorem{rem}[thm]{Remark}

\newtheorem*{claim*}{Claim}

\def\Dax{{\rm Dax}}
\def\Diff{{\rm Diff}}
\def\MCG{{\rm MCG}}
\def\aut{{\rm Aut}}

\def\RR{\mathbb{R}}

\def\ZZ{\mathbb{Z}}

\def\S{\mathscr{S}}
\def\F{\mathscr{F}}

\begin{document}

\title[On the Dax invariants of $S^2$-bundles over surfaces]{On the Dax invariants of $S^2$-bundles over surfaces}

\author{Qizheng You}

\address[Q. You]{School of Mathematical Sciences, Peking University, China.}
\email{qizhengyou@stu.pku.edu.cn}

\date{}

\thanks{}

\begin{abstract}
This paper studies the nontrivial orientable $S^2$-bundle $\Sigma \ltimes S^2$ over a closed surface $\Sigma$ of genus $g \geq 1$.
We have three main results as follows.
We construct the relative Dax invariants for pointed embeddings of $\Sigma$ into $\Sigma \ltimes S^2$, 
which satisfies isotopy invariance, additivity, and naturality.
For some embedded surfaces in $\Sigma \ltimes S^2$ with a fixed geometric dual, 
we establish a complete classification up to isotopy.
We give an alternative construction of a surjective homomorphism $\hat{\Phi}: \MCG(\Sigma \ltimes S^2) \to \mathbb{Z}^\infty$.

\end{abstract}

\date{\today}
\maketitle

\parskip 4pt


\section{Introduction}\label{sec:intro}

In the seminal work \cite{G19}, 
David Gabai established the 4-dimensional Light Bulb Theorem, 
while introducing obstructions 
that preclude homotopic embedded closed surfaces from being isotopic. 
In \cite{G21}, 
Gabai studied the analogous problem for neatly embedded disks in 4-manifolds with a common boundary geometric dual,
and provided a new interpretation of the Dax invariants for embedded disks, 
which were first introduced in \cite{D72}. 
Subsequently in \cite{KT24}, Kosanović and Teichner showed that,
for neatly embedded disks in smooth 4-manifolds that are homotopic relative to the boundary,
the Dax invariants are the only obstruction to isotopy.
These ‘‘homotopy does not imply isotopy’’ phenomena and the Dax invariants, 
have garnered much scholarly attention,
since they have relevance to the investigations into the topology of smooth 4-manifolds (see, e.g. \cite[Corollary 1.7]{G19}, \cite{G20} ,and \cite[Theorem 2.3]{KT24}).

Later in the work \cite{LWXZ25}, 
Lin, Wu, Xie and Zhang extend the construction of the Dax invariants to embedded closed surfaces in the 4-manifolds $\Sigma \times S^2$,
where $\Sigma$ is the closed surface with genus $g \geq 1$.
They apply the Dax invariants to the study of 
the mapping class group, the isotopy classification for embedded surfaces, and the nonisotopic symplectic structures.
They remark that their idea when defining the Dax invariants 
in principle applies to the embeddings of closed surfaces in arbitrary 4-manifolds.
In this paper, we will construct the Dax invariants for embedded surfaces in the nontrivial orientable $S^2$-bundle over $\Sigma$,
under the methodological guidance of \cite{LWXZ25}.

In what follows, we work exclusively within the smooth category.
Let $\Sigma$ be the closed surface with genus $g \geq 1$.
Let $\Sigma \ltimes S^2$ be the orientable $S^2$-bundle over $\Sigma$, 
required that its second Stiefel-Whitney class is nonzero.
Our first main result is defining the relative Dax invariants for some embedded surfaces in $\Sigma \ltimes S^2$.
We construct a family of pointed embeddings $\{\S_d: \Sigma \hookrightarrow \Sigma \ltimes S^2\}_{d \in \ZZ}$ in Definition \ref{defofsd}.
Let $\mathcal{E}_d$ be the set consisting of the pointed embeddings $i: \Sigma \hookrightarrow \Sigma \ltimes S^2$ homotopic to $\S_d$ relative to the basepoint.
Let $\mathcal{C}$ be the set consisting of the conjugacy classes of $\pi_1(\Sigma \ltimes S^2)$ except the conjugacy class containing the unit.
Denote the free abelian group generated by $\mathcal{C}$ as $\ZZ[\mathcal{C}]$.
There is an involution $\sigma$ on $\ZZ[\mathcal{C}]$, sending $[g]$ to $[g^{-1}]$.
Denote the invariant subgroup of $\ZZ[\mathcal{C}]$ under the action of $\sigma$ as $\ZZ[\mathcal{C}]^\sigma$. 
We have the following theorem.

\begin{thm}\label{thm1}
  There is a map 
  $$ \Dax : \mathcal{E}_d \times \mathcal{E}_d \rightarrow \ZZ[\mathcal{C}]^\sigma, $$
  satisfying the following properties.
  \begin{itemize}
    \item[(1)](Isotopy invariance) If $i,i' \in \mathcal{E}_d$ are isotopic relative to the basepoint, $\Dax(i,i') = 0$.
    \item[(2)](Additivity) It holds that $\Dax(i_1,i_2) + \Dax(i_2,i_3) = \Dax(i_1,i_3)$.
    \item[(3)](Naturality) For any pointed diffeomorphism $f: \Sigma \ltimes S^2 \rightarrow \Sigma \ltimes S^2$ and $i_1,i_2 \in \mathcal{E}_d$, it holds that
  $ \Dax(f \circ i_1,f \circ i_2) = f_* \circ \Dax(i_1, i_2). $
  Here $f_*$ is the induced automorphism of $f$ on $\ZZ[\mathcal{C}]^\sigma$.
  \end{itemize}
\end{thm}

We briefly sketch the construction of the Dax invariants for embeddings of surfaces in $\Sigma \ltimes S^2$.
We may fix a handle decomposition of $\Sigma$, such that the 0-handle $H_0$ containing the basepoint $b_0$ of $\Sigma$.
Given $ i \in \mathcal{E}_d$, we can apply an isotopy such that the restrictions of $i$ and $\S_d$ to the 0-handle $H_0$ and 1-handles $H_1$ are identical.
Let $M_2$ be the manifold obtained by excising a tubular neighbourhood of $\S_d(H_0 \cup H_1)$  as Definition \ref{m1m2}.
We need a further isotopy to ensure that the restrictions of $i$ and $\S_d$ to the 2-handle $H_2$,
as embedded disks in $M_2$, are homotopic relative to the boundary.
Then we can define $\Dax(i,\S_d)$ to be the image of $\Dax_{M_2}(\langle i|_{H_2},\S_d|_{H_2} \rangle)$ under the canonical projection 
$$\ZZ[\pi_1(M_2)] \cong \ZZ[\pi_1(\Sigma \ltimes S^2)] \rightarrow \ZZ[\mathcal{C}].$$
Here $\langle-,-\rangle$ refers to the the scanning map as (\ref{scmap}),
and $\Dax_{M_2}$ refers to the Dax isomorphism of $M_2$. 
See Subsection \ref{subsec:daxofm2} for a detailed description for $\Dax_{M_2}$.
Some additional clarification is required to establish that the Dax invariants is well-defined.
Then Theorem \ref{thm1}(1)(2) are implied by the definition, and Theorem \ref{thm1}(3) is shown in Lemma \ref{naturality}.

Our second main result focuses on the isotopy classification of embedded surfaces in $\Sigma \ltimes S^2$ with the common geometric dual $S'_0$.
Here $S'_0$ is the $S^2$-fiber of $\Sigma \ltimes S^2$ over the basepoint $b_0$ of $\Sigma$.
\begin{defn}\label{deffd}
  Let $\mathcal{F}_d$ be the set consisting of the embedded surfaces in $\Sigma \ltimes S^2$, satisfying the following condition.
  \begin{itemize}
    \item[(1)] Any embedded surfaces $\Sigma \in \mathcal{F}_d$ intersect $S'_0$ transversely and positively in $\Sigma \ltimes S^2$ at one point denoted as $b$.
    \item[(2)] The algebraic intersection number of any $\Sigma \in \mathcal{F}_d$ with $\S_{-1}$ is identical to $d$.
    \item[(3)] Any embedded surfaces $\Sigma_1 \in \mathcal{F}_d$ agree with the embedded surface $\S_d(\Sigma)$ on a neighbourhood of $S'_0$.
  \end{itemize} 
  There is an equivalence relation $\approx$ on $\mathcal{F}_d$. 
  For $\Sigma_1, \Sigma_2 \in \mathcal{F}_d$, say $\Sigma_1 \approx \Sigma_2$, 
  if there are parameterizations $i_1,i_2$ for $\Sigma_1,\Sigma_2$ such that $i_1(b_0) = i_2(b_0) = b$,
  and $i_1$ and $i_2$ are isotopic relative to $b_0$.
\end{defn}

Then we have the following theorem.

\begin{thm}\label{thm3}
 There is a bijection
 $$ \Lambda : \mathcal{F}_d/\approx  \ \rightarrow \ \ZZ[\mathcal{C}]^\sigma, \quad [\Sigma] \mapsto \Dax(i,\S_d). $$
 Here $i$ is a parameterization for $\Sigma$ such that $i(b_0) = b$ and $i$ is homotopic to $\S_d$ relative to the basepoint $b_0$.
\end{thm}

The surjectivity of $\Lambda$ relies on the construction of self-referential tubes on an embedded surface.
The injectivity of $\Lambda$ is based on an analogy of the isotopy procedure when defining the Dax invariants in Theorem \ref{thm1},
and the results in \cite[Theorem 1.1]{KT24} of Kosanović and Teichner. 

Let $\Diff(\Sigma \ltimes S^2)$ be the group consisting of the self-diffeomorphisms of $\Sigma \ltimes S^2$.
Define the mapping class group $\MCG(\Sigma \ltimes S^2) = \pi_0 (\Diff(\Sigma \ltimes S^2))$.
There is a subgroup $\MCG_0(\Sigma \ltimes S^2)$ of $\MCG(\Sigma \ltimes S^2)$, 
consisting of the mapping classes $[f]$ satisfying that $f$ is homotopic to the identity map.
As another application of the Dax invariants in $\Sigma \ltimes S^2$, we give an alternative proof of the following theorem of Guo in \cite{Guo25}.

\begin{thm}\label{thm2}
  There is a surjective homomorphism
  $$ \hat{\Phi}: \MCG(\Sigma \ltimes S^2) \rightarrow \ZZ^\infty. $$
  The image of $\MCG_0(\Sigma \ltimes S^2)$ under the homomorphism $\hat{\Phi}$ is also of infinite rank.
\end{thm}

Theorem \ref{thm2} has a direct corollary that the mapping class group of $\Sigma \ltimes S^2$ is infinitely generated.
In \cite{Guo25}, 
Guo constructed a surjective homomorphism $\Phi': \MCG(\Sigma \ltimes S^2) \rightarrow \ZZ^\infty$,
by lifting a self-diffeomorphism of $\Sigma \ltimes S^2$ to its finite covering spaces and 
utilizing the Dax invariants in \cite{LWXZ25} for $\Sigma \times S^2$.
We give another construction of such surjective homomorphism in this paper, by the Dax invariants for $\Sigma \ltimes S^2$. 
The proof of Theorem \ref{thm2} relies on 
the Dax invariants in Theorem \ref{thm1},
the notions of self-referential disks in \cite{G21},
and the self-referential barbell diffeomorphisms of 4-manifolds.

The paper is organized as follows.
In Section \ref{sec:pre}, 
we review some notions about the Dax invariants for disks and the barbell diffeomorphisms,
construct the 4-manifolds $\Sigma \ltimes S^2$, and show some properties of $\Sigma \ltimes S^2$ as well.
In Section \ref{sec:dax},
we construct the Dax invariants for some embeddings of surfaces in $\Sigma \ltimes S^2$ and prove Theorem \ref{thm1}.
In Section \ref{sec:isoclass},
we give the isotopy classification of surfaces in $\mathcal{F}_d$ via the Dax invariants and prove Theorem \ref{thm3}.
In Section \ref{sec:mcg},
we study the mapping class group of $\Sigma \ltimes S^2$ via the Dax invariants 
and prove Theorem \ref{thm2}.

\subsection*{Acknowledgement} 
The author would like to thank his advisor Yi Xie for his continuous guidance and assistance throughout the entire research process.

\section{Preliminaries}\label{sec:pre}

\subsection{The Dax invariants for disks}\label{subsec:diskdax}

In this subsection, we review the definition of the Dax invariants for embedded disks in 4-manifolds with boundaries in \cite{G21}.

We recall the definition of neat maps.

\begin{defn} 
  A smooth map $i:X \rightarrow Y$ between two smooth manifolds with boundaries is neat, 
  if $i^{-1}(\partial Y) = \partial X$, and $i$ is transverse to $\partial Y$.
\end{defn}

Then we provide some notations of the embedding spaces necessary in what follows.

\begin{defn}\label{emb1}
  Let $X,Y$ be smooth manifolds with boundaries.
  \begin{itemize}
    \item[(1)] Define ${\rm Emb}(X,Y)$ to be the space consisting of all the neat embeddings $i$ of $X$ into $Y$, with $C^\infty$ topology.
    \item[(2)] Given $i_0 \in {\rm Emb}(X,Y)$, define ${\rm Emb}_\partial (X,Y)$ to be the subspace of ${\rm Emb}(X,Y)$ 
    consisting of the embeddings $i:X \rightarrow Y$ satisfying $i|_{\partial X} = i_0 |_{\partial X}$.
    \item[(3)] Suppose ${\rm codim}(X,Y) \geq 1$. Given $i_0 \in {\rm Emb}(X,Y)$, define ${\rm Emb}'_{\partial}(X,Y)$ to be the space
    consisting of the embeddings in ${\rm Emb}_\partial (X,Y)$, equipped with a nowhere vanishing section of its normal bundle.
    \item[(4)] Suppose ${\rm codim}(X,Y) \geq 1$. Given $i_0 \in {\rm Emb}(X,Y)$, define ${\rm Emb}^{\rm Fr}_{\partial}(X,Y)$ to be the space
    consisting of the embeddings in ${\rm Emb}_\partial (X,Y)$, equipped with a frame of its normal bundle.
  \end{itemize}
\end{defn}

\begin{defn}\label{emb2}
  Let $X,Y$ be smooth manifolds with basepoints. Let $x_0,y_0$ be the basepoints of $X,Y$ respectively.
  \begin{itemize}
    \item[(1)] Define ${\rm Emb}_\bullet (X,Y)$ to be the space consisting of all the embeddings $i$ of $X$ into $Y$ satisfying $i(x_0)=y_0$.
    \item[(2)] Suppose $X'$ is a submanifold of $X$ of codimension 0 containing $x_0$. Given $i_0 \in {\rm Emb}_\bullet(X,Y)$, 
    define ${\rm Emb}^{[i_0]}_\bullet (X',Y)$ to be the space consisting of all the embeddings $i$ of $X'$ into $Y$ satisfying $i(x_0)=y_0$,
    and $i|_{X'}$ is homotopic to $i_0|_{X'}$ relative to $x_0$.
  \end{itemize}
\end{defn}

Let $X$ be an oriented 4-manifold with boundaries.
Let $I = [0,1]$.
We review the Dax isomorphism in 4-manifolds following \cite{G21}.
Let $I_0$ be an element of ${\rm Emb}(I,X)$.
We have a tracing homomorphism
\begin{equation}\label{tracemap}
  \mathscr{F} : \pi_1({\rm Emb}_\partial (I,X); I_0) \rightarrow \pi_2(X;I_0).
\end{equation} 
Here $\pi_2(X;I_0)$ is the second homotopy group of $X$ relative to the image of $I_0$.
The tracing homomorphism maps the homotopy class of a loop $\{\gamma_t\}_{t \in [0,1]}$ to the homotopy class of $\gamma_-(-):I \times I \rightarrow X$ in $\pi_2(X;I_0)$.
Denote the kernel of $\mathscr{F}$ as $\pi_1^D({\rm Emb}_\partial (I,X); I_0)$.
By the Dax isomorphism theorem \cite[Theorem 0.2]{G21}, there is an abelian group homomorphism
$$ d_3: \pi_3(X;I_0) \rightarrow \ZZ[\pi_1(X;I_0) \setminus \{1\}], $$
and an abelian group isomorphism
$$ \Dax: \pi_1^D({\rm Emb}_\partial (I,X); I_0) \rightarrow \ZZ[\pi_1(X;I_0) \setminus \{1\}]/{\rm Im}(d_3). $$

Here we review the definition of the homomorphism $\Dax$.
Given an element $\alpha$ in $\pi_1^D({\rm Emb}_\partial (I,X); I_0)$,
it has a representative $\gamma_t : I \rightarrow {\rm Emb}_\partial (I,X)$ with endpoints $I_0$.
There is a natural inclusion ${\rm Emb}_\partial (I,X) \rightarrow {\rm Imm}_\partial (I,X)$,
where ${\rm Imm}_\partial (I,X)$ is the space consisting of immersions $i$ from $I$ to $X$ with the same endpoints as $I_0$. 
The loop $\gamma_t : I \rightarrow {\rm Imm}_\partial (I,X)$ is null homotopic.
By a transversal argument, we may select the homotopic $H(t,s): I \times I \rightarrow {\rm Imm}_\partial (I,X)$ to satisfy the following conditions.
\begin{itemize}
  \item The path $H(t,0)$ is the constant path $c_{I_0}$. The path $H(t,1)$ is identical to $\gamma_t$.
  \item The homotopy $H(t,s)$ maps $I \times I$ into ${\rm Emb}_\partial (I,X)$ except for finitely many points in $I \times I$.
  \item Denote the except points in $I \times I$ as $(t_1,s_1),\cdots,(t_n,s_n)$.
  For $1 \leq j \leq n$, the element $H(t_j,s_j)$ is an immersion in ${\rm Imm}_\partial (I,X)$ with only one double point.
\end{itemize}
For $1 \leq j \leq n$, denote $\iota_j = H(t_j,s_j)$. 
Then there are $x_1 < x_2 \in I$ satisfying $\iota_j(x_1) = \iota_j(x_2)$.
Choose the basepoint of $X$ to be the starting point of $I_0$.
The concatenation of $\iota_j|_{[0,x_2]}$ and the reverse of $\iota_j|_{[0,x_1]}$ gives an element in $\pi_1(X)$.
Denote the element in $\pi_1(X)$ as $g_j$.
The orientation of $X$ and the local orientation $${\rm d}H(-,-)(x_1)|_{(t_j,s_j)} \oplus {\rm d}H(-,-)(x_2)|_{(t_j,s_j)}$$
give a sign $\epsilon_j \in \{1,-1\}$.
Then we define
\begin{equation*}
  \Dax(\alpha) = \sum_{j=1}^{n} \epsilon_j g_j \in \ZZ[\pi_1(X;I_0) \setminus \{1\}]/{\rm Im}(d_3).
\end{equation*}

We review the relative Dax invariants for embedded disks following \cite{G21}. 
Let $D_0$ be a neat embedding of 2-disk in $X$. 
By the assumption of neatness, 
the embedding space ${\rm Emb}_\partial (D^2,X)$ is homotopic to the its subspace $E$ consisting of the embedding $i: D^2 \hookrightarrow X$ 
that coincides with $D_0$ in a neighbourhood $N(\partial D^2)$ of $\partial D^2$.
We select the neatly embedded arc $I_0$ to lie in the image of $N(\partial D^2)$ under $D_0$.
For $i \in \pi_0(E)$, we obtain a path $\{I_t\}$ in ${\rm Emb}_\partial (I,X)$ given by scanning $i_1(D^2)$ from $I_0$ to another arc $I_1$ in $D_0(N(\partial D^2))$.
For $i_1, i_2 \in \pi_0(E)$, 
since they coincide on $N(\partial D^2)$, we may assume that the scanning families of $i_1$ and $i_2$ have the same endpoints.
Then we can concatenate the scanning paths of $i_1$ and $i_2$ to obtain a loop in ${\rm Emb}_\partial (I,X)$ with basepoint $I_0$.
Denote the homotopy class of the loop as $\langle i_1 , i_2 \rangle$.
Then we get the scanning map
\begin{equation}\label{scmap}
  \langle - , - \rangle: \pi_0({\rm Emb}_\partial (D^2,X)) \times \pi_0({\rm Emb}_\partial (D^2,X)) \rightarrow \pi_1({\rm Emb}_\partial (I,X); I_0).
\end{equation}
For neat embeddings $i_1, i_2$ of disks in $X$, assume that they are homotopic relative to the boundaries.
Then we have $\langle i_1 , i_2 \rangle \in \pi_1^D({\rm Emb}_\partial (I,X))$.
We define the Dax invariant
$$ \Dax (i_1,i_2) = \Dax(\langle i_1 , i_2 \rangle) \in \ZZ[\pi_1(X;I_0) \setminus \{1\}]/{\rm Im}(d_3).$$

\begin{rem}
  Since the self-diffeomorphism group of $D^2$ relative to the boundary is contractible according to \cite[Example 6.5.2]{M16},
  the Dax invariant $\Dax(i_1,i_2)$ only depends on the image of $i_1$ and $i_2$.
\end{rem}

\begin{rem}
  According to \cite[Lemma 1.4]{K24},
  there is an involution $\sigma$ on $\ZZ[\pi_1(X)\setminus \{1\}]$ sending $g$ to $g^{-1}$.
  It holds that $\sigma({\rm Im} d_3) = {\rm Im} d_3$.
  Thus, the action of $\sigma$ on $\ZZ[\pi_1(X) \setminus \{1\}]$ induces its action on $\ZZ[\pi_1(X) \setminus \{1\}]/ {\rm Im} d_3$.
  Actually, the codomain of the Dax invariants for disks is $(\ZZ[\pi_1(X) \setminus \{1\}]/ {\rm Im} d_3)^\sigma$,
  namely, the invariant subgroup of $\ZZ[\pi_1(X) \setminus \{1\}]/ {\rm Im} d_3$ under the action of $\sigma$.
\end{rem}

\subsection{The barbell diffeomorphisms}\label{bbldiff}

In this subsection, we review the barbell diffeomorphisms of 4-manifolds in \cite{G19}.

\begin{defn}
  Let $M$ be a smooth manifold may with boundaries.
  Let $\Diff_\partial (M)$ be the group consisting of self-diffeomorphisms $f$ of $M$, satisfying that $f|_{\partial M} = {\rm id}_{\partial M}$.
  We define the mapping class group of $M$ 
  $$\MCG_\partial (M) = \pi_0 (\Diff_\partial (M)).$$
  If $M$ is a smooth manifold without boundary, we may denote $\Diff_\partial (M), \MCG_\partial(M)$ as $\Diff(M)$ and $\MCG(M)$ for simplicity. 
\end{defn}

Let $\mathcal{B} = (S^2 \times D^2) \natural (S^2 \times D^2)$, namely, the boundary connected sum of two copies of $S^2 \times D^2$.
We call $\mathcal{B}$ the standard barbell.
We call the first (resp. second) copy of $S^2 \times D^2$ as the first (resp. second) cuff of the standard barbell.
Regarding the first cuff as the complement of a tubular neighbourhood $\nu(I_0)$ of a neatly embedded arc $I_0$ embedded in $D^4$,
we obtain that $\mathcal{B} = D^4 \natural (S^2 \times D^2) \setminus \nu(I_0)$.
We swipe the arc $I_0$ around the second cuff to get an element $I_t$ in $\pi_1({\rm Emb}^{\rm Fr}_{\partial}(I,D^4 \natural (S^2 \times D^2));I_0)$.
Under the arc-pushing map 
$$ \partial: \pi_1({\rm Emb}^{\rm Fr}_{\partial}(I,D^4 \natural (S^2 \times D^2));I_0) \rightarrow \pi_0({\rm Diff}_\partial(D^4 \natural (S^2 \times D^2) \setminus \nu(I_0))) = \MCG_\partial(\mathcal{B}),$$
the image $\beta$ of $I_t$ is what we call the standard barbell diffeomorphism.
We may abuse the notation $\beta$ and call a representative of $\beta$ the standard barbell diffeomorphism when there is no ambiguity.

Let $D_0 = \{*\} \times D^2$ be a neatly embedded disk in $S^2 \times D^2$.
We can regard $D_0$ as a neatly embedded disk in the first cuff of the standard barbell.
Then the standard barbell diffeomorphism $\beta$ can act on the disk $D_0$.
According to \cite[Construction 5.3]{G19},
the disk $\beta(D_0)$ is obtained by tubing a copy of the embedded sphere $S^2 \times \{*\}$ 
in the second cuff to $D_0$ along the bar connecting the cuffs.

\begin{defn}
  Let $X$ be a 4-manifolds, and $A,B$ be subsets of $X$.
  Define $\hat{\pi}_1(X;A,B)$ to be the set consisting of the connected components of the set 
  $$\{ \text{path} \ \gamma:I \rightarrow X \ | \ \gamma(0) \in A, \gamma(1) \in B, \gamma(0,1) \cap (A \cup B)= \emptyset \}.$$
\end{defn}

Let $X$ be a 4-manifolds. 
The defining data for a barbell diffeomorphism of $X$ consists of 
\begin{itemize}
  \item two embedded sphere $S_0, S_1$ with trivial normal bundles,
  \item an element $\gamma \in \hat{\pi}_1(X;S_0,S_1)$.
\end{itemize}
These data defines a barbell diffeomorphism $\beta_{(S_0,S_1,\gamma)}$ of $X$.
In what follows, we introduce the construction of such barbell diffeomorphisms. 

Given the defining data $(S_0,S_1,\gamma)$.
The closed regular neighbourhood $\bar{\nu}(S_0 \cup \gamma \cup S_1)$ is diffeomorphic to $\mathcal{B}$.
We fix a diffeomorphism $\varphi : \mathcal{B} \rightarrow \bar{\nu}(S_0 \cup \gamma \cup S_1)$.
We extend the self-diffeomorphism
$ \varphi \circ \beta \circ \varphi^{-1}$ on $\bar{\nu}(S_0 \cup \gamma \cup S_1)$
by identity on $X \setminus \bar{\nu}(S_0 \cup \gamma \cup S_1)$.
Then we obtain the implemented barbell diffeomorphism
$$ \beta_{(S_0,S_1,\gamma)}: X \rightarrow X.$$
\begin{rem}
  According to \cite[Remark 2.5]{LWXZ25}, different choices of the diffeomorphisms $\varphi$ do not change the isotopy class of $\beta_{(S_0,S_1,\gamma)}$ in $\MCG(X)$.
\end{rem}

\subsection{The nontrivial orientable $S^2$-bundle over $\Sigma$}\label{subsec:tstspace}
In this subsection, we study the space $\Sigma \ltimes S^2$ and embedded surfaces in $\Sigma \ltimes S^2$.

Let $\Sigma$ be an orientable closed surface with genus $g \geq 1$. 
We consider orientable $S^2$-fiber bundles over $\Sigma$.
Up to bundle isomorphism, there are exactly two $S^2$-fiber bundles over $\Sigma$, namely, the trivial bundle $\Sigma \times S^2$ and the twisted one $\Sigma \ltimes S^2$.
These bundles are characterized by their second Stiefel-Whitney classes $w_2 \in H^2(\Sigma;\ZZ_2)$.
Moreover, as smooth 4-manifolds, $\Sigma \times S^2$ and $\Sigma \ltimes S^2$ are not diffeomorphic,
since their intersection forms are not isomorphic.
 
Since $\pi_1 (SO(3)) = \ZZ / 2$, we can choose $\phi : S^1 \rightarrow SO(3)$ such that the homotopy class of $\phi$ is nontrivial.
Let $D$ be a disk in $\Sigma$.
We can construct $\Sigma \ltimes S^2$ by defining 
\begin{equation}\label{construct}
  \Sigma \ltimes S^2 = (\Sigma-D)\times S^2 \cup_\phi D\times S^2.
\end{equation}
To be precise, we identify $\partial (\Sigma - D) = \partial D = S^1$, 
then we attach $(\Sigma-D)\times S^2$ and $D\times S^2$ by identifying $(x,y)$ and $(x,\phi(x)y)$ for any $x \in S^1, y \in S^2$.

We compute the homotopy groups and the cohomology of $ \Sigma \ltimes S^2$.
By the homotopy long exact sequences induced by fibrations, 
the induced map of the projection $\Sigma \ltimes S^2 \rightarrow \Sigma$ on the fundamental groups is an isomorphism.
The induced maps of the fiber inclusion $S^2 \rightarrow \Sigma \ltimes S^2$ on higher homotopy groups are also isomorphisms. 
Therefore, it holds that 
$$ \pi_1 (\Sigma \ltimes S^2) \cong \pi_1(\Sigma); \quad \pi_k(\Sigma \ltimes S^2) \cong \pi_k(S^2), \text{for} \ k \geq 2.$$
We consider the Serre spectral sequence for the fiber bundle $\Sigma \ltimes S^2$ with coefficients in $\ZZ$.
Because $\Sigma \ltimes S^2$ is an orientable $S^2$-fiber bundle,
the fundamental group of the base space $\Sigma$ acts trivially on the cohomology of the fiber $S^2$.
Therefore, the $E_2$-page satisfies that $E_2^{p,q} = H^p(\Sigma,H^q(S^2,\mathbb{Z}))$. 
By the dimensional restriction, all differentials on the $E_2$-page must be trivial.
We obtain the following group isomorphism
\begin{equation}\label{cohom}
  H^*(\Sigma \ltimes S^2) \cong H^*(\Sigma) \otimes H^*(S^2).
\end{equation}
By a geometric computation, the intersection form of $\Sigma \ltimes S^2$ is
\begin{equation*}
\begin{pmatrix}
1 & 1 \\
1 & 0
\end{pmatrix},
\end{equation*}
with respect to the basis $(\alpha \otimes 1, 1 \otimes \beta)$ of $H^2(\Sigma \ltimes S^2)$ as in the isomorphism (\ref{cohom}).
Here $\alpha,\beta$ are generators of $H^*(\Sigma), H^*(S^2)$ respectively.

We fix a cellular structure of $\Sigma$. 
It has one 0-cell $e^0$, $2g$ 1-cells $e^1_1, \cdots, e^1_{2g}$, and one 2-cell $e^2$.
The cellular structure induces a handle decomposition of $\Sigma$.
It has one 0-handle $H_0$, $2g$ 1-handles $H_1^1, \cdots, H_1^{2g}$, and one 2-handle $H_2$.
We require that the 2-disk $D$ in (\ref{construct}) is a smoothly embedded disk in the two handle $H_2$.
Let the basepoint of $\Sigma$ be $b_0 = e^0$. 

Next, we construct several pointed embeddings of $\Sigma$ in $\Sigma \ltimes S^2$. 
To facilitate the construction, we fix a model for $\Sigma \ltimes S^2$.
Let the nontrivial loop $\phi: S^1 \rightarrow SO(3)$ in (\ref{construct}) to be
\begin{equation}\label{model}
  \phi(\theta) = \begin{pmatrix}
  \cos \theta & \sin \theta & 0 \\
  -\sin \theta & \cos \theta & 0 \\
  0 & 0 & 1
  \end{pmatrix}.
\end{equation}

\begin{defn}\label{defofsd}
  Define the embedding
  $$\S_0: \Sigma \hookrightarrow \Sigma \ltimes S^2, \quad x \mapsto (x,(0,0,1)^T).$$ 
  It is well defined because $(0,0,1)^T$ is a fixed point in $S^2$ under the action of $\phi(\theta)$ for all $\theta \in S^1$.  
  Given a postive integer $d$, we choose $d$ distinct points $p_1,\cdots,p_d$ in $\Sigma$, requiring that they are in the 2-handle $H_2$. 
  Let $\S_d: \Sigma \rightarrow \Sigma \ltimes S^2$ to be the embedding obtained by 
  resolving the intersections of $\S_0$ and the $S^2$-fibers of $\Sigma \ltimes S^2$ over $p_1,\cdots,p_d$. 
  Given a negative integer $d$, we define $\S_d$ in the same way, 
  but resolve the intersections of $\S_0$ and the $S^2$ fibers over $p_1,\cdots,p_{-d}$ with the opposite orientation.
\end{defn}

Let the basepoint of $\Sigma \ltimes S^2$ to be $b = \S_0(b_0)$. 
Now we get a family $\{\S_d\}_{d \in \ZZ}$ of pointed embeddings of $\Sigma$ in $\Sigma \ltimes S^2$.

\begin{rem}
  According to Proposition \ref{bijmap}, 
  any pointed map $f: \Sigma \rightarrow \Sigma \ltimes S^2$ inducing the identity homomorphism on the fundamental groups,
  is homotopic to $\S_d$ for some $d \in \ZZ$.
\end{rem}

\section{Definition of the Dax invariants for embedded surfaces}\label{sec:dax}

In this section, we are devoted to defining the Dax invariants for the embedded surfaces in $\Sigma \ltimes S^2$ we constructed in Definition \ref{defofsd}, 
following the procedure of \cite{LWXZ25}.
For simplicity, we denote $\Sigma \ltimes S^2$ as $M$ in the following content.

We remark that many arguments in \cite[Section 3 and Section 4]{LWXZ25} are applicable to the setting of this paper. 
We provide only a brief overview of the applicable arguments, 
while presenting a detailed account of the cases requiring modification.
The key modifications are given in Lemma \ref{fiberhomotopy}(3), Proposition \ref{bijmap}, Lemma \ref{K0} and Lemma \ref{naturality}. 

\subsection{The fibration tower of embedding spaces}

There is a fibration tower 
\begin{equation}\label{fbtw}
  {\rm Emb}^{[\S_d]}_\bullet (\Sigma,M) \xrightarrow{r_2}  {\rm Emb}^{[\S_d]}_\bullet (H_0\cup H_1,M) \xrightarrow{r_1} {\rm Emb}^{[\S_d]}_\bullet (H_0,M) \xrightarrow{r_0} *.
\end{equation}
useful for defining the Dax invariants. 
For $i \in {\rm Emb}^{[\S_d]}_\bullet (\Sigma,M)$, $r_2(i)$ is the restriction of $i$ to the submanifold $H_0 \cup H_1 \subset \Sigma$.
Similarly, the maps $r_1,r_0$ are defined by restrictions of the domains.
Denote the domain of $r_j$ as $E_j$, for $j=0,1,2$. 
Let $F_j$ be the fiber of $r_j$, for $j=0,1,2$.

Let $U$ be a disk containing $(0,0,1)^T$ in $S^2$. 
In the model of $M$ constructed in (\ref{construct}) and (\ref{model}), 
let $\nu(\S_d(e^0)) = H_0 \times U$, $\nu(\S_d(e_j^1))= (H_0 \cup H_1^j) \times U$.

\begin{defn}\label{m1m2}
  Let \( M_1 \) be the closure of \( M \setminus \nu(\S_d(e^0)) \). 
  Let \( M_2 \) be the closure of \( M \setminus \left( \cup_{j=1}^{2\ell} \nu(\S_d(e_j^1)) \right) \). 
  We smooth corners and regard both \( M_1 \) and \( M_2 \) as codimension-0 submanifolds of \( M \).
\end{defn}

\begin{defn}
  We identify $H_1 \cap {\rm sk}^1$ to $\sqcup_{2g} I$. 
  For $\S_d |_{\sqcup_{2g} I} \in {\rm Emb}(\sqcup_{2g}I,M_1)$, 
  define the embedding space ${\rm Emb}'_\partial(\sqcup_{2g} I,M_1)$ as Definition \ref{emb1}(3).
  There is a nonvanishing section $\eta:\sqcup_{2g}I \rightarrow T\Sigma$ of the normal bundle of the embedded arcs $\sqcup_{2g}I$ in $\Sigma$.
  We equip the embedding $\S_d |_{\sqcup_{2g} I}$ with the frame $({\rm d}\S_d) \circ \eta$,
  and denote it as $\S_d^{(1)}$.
  Note that $\S_d^{(1)}$ is an element of ${\rm Emb}'_\partial (\sqcup_{2g} I,M_1)$.
\end{defn}

\begin{defn}
  For $\S_d|_{H_2} \in {\rm Emb}(D^2,M_2)$,
  define the embedding space ${\rm Emb}_\partial(D^2,M_2)$ as Definition \ref{emb1}(2).
  Denote ${\rm Emb}^{\rm hom}_\partial (D^2,M_2)$ to be the subspace of ${\rm Emb}_\partial(D^2,M_2)$,
  consisting of the embeddings $i$ in ${\rm Emb}_\partial(D^2,M_2)$ such that 
  $$i_* = (\S_d|_{H_2})_* : H_2(D^2,\partial D^2) \rightarrow H_2(M_2,\partial M_2).$$ 
  Here $i_*$ and $(\S_d|_{H_2})_*$ are induced homomorphisms of $i$ and $\S_d|_{H_2}$ on the homology.
\end{defn}

In what follows, we want to show the homotopy types of the fibers.

\begin{lem}\label{fiberhomotopy}
  It holds the following.
  \begin{itemize}
    \item[(1)] The fiber $F_0$ is homotopy equivalent to the real Grassmannian manifold $V_{2,4}$.
    \item[(2)] The fiber $F_1$ is homotopy equivalent to the connected component of ${\rm Emb}'_\partial (\sqcup_{2g} I,M_1)$ that contains $\S_d^{(1)}$.
    \item[(3)] The fiber $F_2$ is identical to ${\rm Emb}^{\rm hom}_\partial (D^2,M_2)$.
  \end{itemize}
\end{lem}

The proofs of the first and second terms of Lemma \ref{fiberhomotopy} are the same as those in \cite[Lemma 3.10]{LWXZ25}. 
Let $\langle \Sigma,M \rangle$ be the set containing the pointed homotopy classes from $\Sigma$ to $M$. 
To prove the third term of Lemma \ref{fiberhomotopy}, we need the following proposition to give a specific description of $\langle \Sigma,M \rangle$.

\begin{prop}\label{bijmap}
  There is a bijection $$ \phi: \langle \Sigma,M \rangle \rightarrow {\rm Hom}(\pi_1(\Sigma),\pi_1(M)) \times \pi_2(M).$$
\end{prop}

\begin{proof}
  We first define a map $\phi_1: \langle \Sigma,M \rangle \rightarrow {\rm Hom}(\pi_1(\Sigma),\pi_1(M))$. 
  For a pointed map $f:\Sigma \rightarrow M$, let $\phi_1([f])$ be the induced homomorphism $f_*$ on the fundamental groups. 
  Note that $\phi_1$ is surjective.
  For any $\varphi \in {\rm Hom}(\pi_1(\Sigma),\pi_1(M))$, we can fix a pointed map $f_\varphi : \Sigma \rightarrow M$ such that $\phi_1([f_\varphi]) = \varphi$.
  
  For pointed maps $g_0,g_1:\Sigma \rightarrow M$, if $\phi_1([g_0]) = \phi_1([g_1])$,
  there is a homotopy $H : {\rm sk}^1 \times I \rightarrow M$ such that $g_i(x) = H(x,i)$ for $i=0,1$ and $x \in {\rm sk}^1$.
  Here ${\rm sk}^1$ is the 1-skeleton of the fixed cellular structure of $\Sigma$.
  Since the map ${\rm sk}^1 \rightarrow \Sigma$ is an inclusion of subcomplex, the homotopy $H$ can be extended to $H':\Sigma \times I \rightarrow M$,
  such that $H'(x,0)=g_0(x)$ for all $x \in \Sigma$.
  Denote $g_0'(x) = H'(x,1)$, then $g_0'|_{{\rm sk}^1} = g_1|_{{\rm sk}^1}$.
  Consider the characteristic map $j: D^2 = e^2 \rightarrow \Sigma$. 
  Since $g_0' \circ j |_{\partial D^2} = g_1 \circ j|_{\partial D^2}$, they induce a map $\psi: S^2 \rightarrow M$.
  We are going to show that the homotopy class $[\psi]$ in $\pi_2(M)$ is independent on the choices of the first homotopy $H$.
  Different choices of $H$ induces an element in $\pi_0(\Omega {\rm Map}_*({\rm sk}^1,M))$, 
  where ${\rm Map}_*({\rm sk}^1,M)$ is the space of pointed maps from ${\rm sk}^1$ to $M$.
  We have the following isomorphisms
  $$ \pi_0(\Omega {\rm Map}_*({\rm sk}^1,M)) \cong \pi_1({\rm Map}_*({\rm sk}^1,M)) \cong \langle S^1 \wedge {\rm sk}^1,M \rangle \cong \oplus_{i=1}^{2g} \pi_2(M).$$ 
  For simplicity, we just describe the action of the first summand of  $\oplus_{i=1}^{2g}\pi_2(M)$ on $[\psi]$.
  Identify the generator $\gamma$ of the first summand of  $\oplus_{i=1}^{2g}\pi_2(M)$ by $H': I \times I \rightarrow M$.
  Assume $H(0,t) = H(1,t) = b$, $H(x,0) = H(x,1) = g_0'(\alpha(x))$ for $x,t \in I$, where $\alpha$ is a simple closed curve on ${\rm sk}^1 \subset \Sigma$ representing $a_1 \in \pi_1(\Sigma)$.
  Then $\gamma \cdot [\psi]$ is represented by the element $\psi'$,
  which is induced from $g_0' \circ j, g_1 \circ j$ and two copies of $H$ by gluing together along the boundaries of their domains.
  Notice that $[\psi']=[\psi]$, hence we can write $[\psi] = \phi_2([g_0],[g_1])$.

  Define $\phi([f]) = (\phi_1([f]), \phi_2([f],[f_0])),$ where $f_0$ is the pointed map fixed before such that $\phi_1([f_0])=\phi_1([f])$. 
  From the above arguments, we conclude that f is well-defined and surjective.
  Next, we will prove that f is injective.

  Given pointed maps $g_0,g_1:\Sigma \rightarrow M$ satisfying $g_0|_{{\rm sk}^1} = g_1|_{{\rm sk}^1}$,
  if $\phi_2(g_0,g_1)=[\psi]$ is trivial in $\pi_2(M)$, we can extend $\psi$ to $\tilde{\psi}: D^3 \rightarrow M$.
  To be specific, we identify $D^2$ with $\{(x,y,z) \in \RR^3 : x^2+y^2+z^2 \leq 1\}$. 
  Assume the restrictions of $\tilde{\psi}$ to the upper and lower hemispheres are $g_0 \circ j$ and $g_1 \circ j$, respectively.
  There is a homotopy given by
  $$ G(x,y,t) = \tilde{\psi} (x,y,(1-2t)\cdot \sqrt{1-x^2-y^2}) .$$
  Notice that $G:D^2 \times I \rightarrow M$ factors through $j \times {\rm id}:D^2 \times I \rightarrow \Sigma \times I$.
  Then $G$ induces $G' : \Sigma \times I \rightarrow M$, which is a homotopy between $g_0$ and $g_1$.
\end{proof}

Now we can prove Lemma \ref{fiberhomotopy}(3).

\begin{proof}[Proof of Lemma \ref{fiberhomotopy}(3)]
  Pick $i \in F_2$. 
  We are going to check that $i_*$ maps the generator $\alpha$ of $H_2(D^2,\partial D^2)$ to the same element in $H_2(M_2;\partial M_2)$ as $(\S_d|_{H_2})_*$.
  It holds that
  $$ H_2(M_2;\partial M_2) \cong H^2(M_2) = H^2(\Sigma \vee S^2) \cong H^2(\Sigma) \oplus H^2(S^2) \cong \ZZ \oplus \ZZ.$$
  The first equation holds via the Lefschetz duality.
  The second equation holds because $M_2$ is homotopic equivalent to $\Sigma \vee S^2$.
  The third equation holds via the Mayer-Vietoris sequence.  
  
  Under such isomorphisms, we only need to check
  $$ \langle i_*(\alpha) , \beta \rangle = \langle (\S_d|_{H_2})_*(\alpha) , \beta \rangle, $$
  for any $\beta \in H_2(M_2) \cong H_2(\Sigma) \oplus H_2(S^2) \cong \ZZ \oplus \ZZ$,
  where we use the notation $\langle -,- \rangle$ to refer the pairing $H^2(M_2) \times H_2(M_2) \rightarrow \ZZ$.
  For $\beta = (1,-1)$, there is an embedding $\S: \Sigma \rightarrow M_2$ representing $\beta$.
  Namely, we can take $\S(x) = (x,(0,0,-1)^T)$ the model of $M$ constructed in (\ref{construct}) and (\ref{model}). 
  Then $\langle i_*(\alpha) , \beta \rangle$ equals to the algebraic intersection number $i(D^2) \cap \S$ in $M_2$. 
  After gluing $i$ with $\S_d|_{H_0 \cup H_1}$, it holds that
  \begin{equation}\label{glueagm}
    \begin{split}
      i(D^2) \cap_{M_2} \S &= (i \cup \S_d|_{H_0 \cup H_1}) \cap_M \S \\
      &= \S_d \cap_M \S \\
      &= (\S_d|_{H_2})(D^2) \cap_{M_2} \S = \langle (\S_d|_{H_2})_*(\alpha) \beta \rangle.
    \end{split}
  \end{equation}
  We use $ \cap_{M_2}, \cap_M$ to denote the algebraic intersection number of $M_2$ and $M$.  
  The first and third equation hold because $\S_d|_{H_0 \cup H_1}$ does not intersect $\S$.
  The second equation holds because $i \in F_2$ and the algebraic intersection number is homotopy invariant.
  For $\beta = (0,1)$, the equation holds by a similar argument.

  Pick $i \in {\rm Emb}^{\rm hom}_\partial (D^2,M_2)$. 
  Here we use the notation in the proof of Proposition \ref{bijmap}. 
  Note that $\phi_1(i \cup \S_d|_{H_0 \cup H_1}) = \phi_1(\S_d).$
  Consider $\phi_2(i \cup \S_d|_{H_0 \cup H_1},\S_d) \in \pi_2(M) \cong \ZZ$.
  It can be identify with the algebraic intersection number with $\S$ in $M_2$.
  Therefore, it holds that 
  $$ \phi_2(i \cup \S_d|_{H_0 \cup H_1},\S_d) = (i(D^2) \cap \S) - ((\S_d|_{H_2})(D^2) \cap \S) = 0.$$
  The last equation holds because $i \in {\rm Emb}^{\rm hom}_\partial (D^2,M_2)$. 
  Now we have shown that $\phi([i \cup \S_d|_{H_0 \cup H_1}]) = \phi([\S_d])$, which concludes the proof.
\end{proof}

There is an action $\pi_1(E_1) \rightarrow {\rm Aut}(\pi_0(F_2))$ by lifting loops in $E_1$. 
Because $\pi_0(E_1)$ is trivial, the quotient map $ \pi_0(F_2)/\pi_1(E_1) \rightarrow \pi_0(E_2) $ 
induced by action is a bijection.
Since $\pi_1(E_0)$ is trivial, there is a surjection $\pi_1(F_1) \rightarrow \pi_1(E_1)$.
Hence we obtain that the quotient map $\pi_0(F_2)/\pi_1(F_1) \rightarrow \pi_0(E_2)$ is a bijection.

\begin{rem}
  The domain of the Dax invariants for surfaces is $\pi_0(E_2) \times \pi_0(E_2)$.
  By utilizing the fibration tower (\ref{fbtw}) and the Dax invariants for disks, 
  we can define a map with domain contained in $\pi_0(F_2) \times \pi_0(F_2)$.
  We may check that the map is invariant under the action of $\pi_1(F_1)$. 
  Thus it induces the Dax invariants for surfaces. 
\end{rem}

\subsection{Properties of $\pi_1(F_1)$}

There is an isomorphism
$$ \pi_1(F_1) \cong \pi_1({\rm Emb}'_\partial (\sqcup_{2g} I,M_1)),$$
due to Lemma \ref{fiberhomotopy}.
In this subsection, we will give several properties of the fundamental group of $F_1$.

We review the construction for two families of elements 
in $\pi_1({\rm Emb}'_\partial (\sqcup_{2g} I,M_1))$ following \cite[Definition 3.20, 3.23]{LWXZ25}.

\begin{defn}
  The basepoint $\S_d^{(1)}$ of ${\rm Emb}'_\partial (\sqcup_{2g} I,M_1)$ is induced by the embedding $\S_d|_{H_1}$.
  Namely, $\S_d^{(1)}$ is a section-equipped embedding of the $2g$ disjoint arcs into $M_1$.
  Given $1 \leq j \leq 2g$,
  if we identify the normal sphere bundle of the $j$-th arc with $I_j \times S^2$,
  constructing a section on $I_j$ is equivalent to specifying an element in $\Omega S^2$.
  Since $\pi_1(\Omega S^2) \cong \pi_2(S^2) \cong \ZZ$, we can pick a loop $\sigma$ presenting a generator of $\pi_1(\Omega S^2)$.
  Let $\hat{\xi}_j$ be the element of $\pi_1({\rm Emb}'_\partial (\sqcup_{2g} I,M_1))$,
  fixing the section-equipped embedding of $I_k$ for $k \ne j$, changing the section on $I_j$ by $\sigma$.
\end{defn}

\begin{defn}\label{S0}
  We fix a point $p$ in $H_2 \subset \Sigma$. Denote the $S^2$-fiber of $M$ at $p$ as $S_0$. 
  Given $1 \leq j \leq 2g$ and $\gamma \in \hat{\pi}_1(I;I_j,S_0)$.
  We define an element $\rho_\gamma$ in $\pi_1({\rm Emb}_\partial (\sqcup_{2g} I,M_1))$.
  For $k \ne j$, we fix the embeddings of $I_k$.
  We isotope $I_j$ along $\gamma$ to $S_0$, spin it around $S_0$, then isotope it back along $\bar{\gamma}$.
  Here $\bar{\gamma}$ refers to the orientation-reversing of $\gamma$.
  The loop $\rho_\gamma$ can be, uniquely up to homotopy, 
  lifted to $\tilde{\rho}_\gamma \in \pi_1({\rm Emb}^{\rm Fr}_\partial (\sqcup_{2g} I,M_1))$,
  since $\pi_1(\Omega SO(3)) = \pi_2(SO(3)) = 0$.
  The loop $\tilde{\rho}_\gamma$ of framed embeddings is restricted to a loop $\hat{\rho}_\gamma$ of section-equipped embeddings.
\end{defn}

The map ${\rm Emb}'_\partial (\sqcup_{2g} I,M_1) \rightarrow {\rm Emb}_\partial (\sqcup_{2g} I,M_1)$, 
given by forgetting the section of an embedding $\sqcup_{2g} I \hookrightarrow M_1$, is a fibration.
Its fiber is homotopic to $\prod_{j=1}^{2g} \Omega S^2$.
Thus, it induces a long exact sequence of homotopy groups
\begin{equation}\label{ses1}
  \cdots \rightarrow \pi_1\big(\prod_{j=1}^{2g}\Omega S^2\big) \cong \ZZ^{2g} \rightarrow \pi_1(F_1) \rightarrow \pi_1{\rm Emb}_\partial (\sqcup_{2g} I,M_1) \rightarrow 0.
\end{equation}
The fundamental group of the fiber is generated by $\{\hat{\xi}_j : j=1,2,\cdots,2g\}$.
By the same argument as the proof for \cite[Lemma 3.25]{LWXZ25}, $\pi_1{\rm Emb}_\partial (\sqcup_{2g} I,M_1)$ is generated by $\{\rho_\gamma\}$.
It deduces that $\pi_1(F_1)$ is the abelian group generated by elements of the form $\hat{\xi}_j$ and $\hat{\rho}_\gamma$.

The fundamental group of $F_1$ acts on $\pi_0(F_2) = \pi_0({\rm Emb}^{\rm hom}_\partial (D^2,M_2))$ by isotopy extension.
Given $\alpha \in \pi_1(F_1)$, $\alpha$ can be represented by an isotopy $f_t:H_1 \rightarrow M_1$. 
By isotopy extension and evaluating the ambient isotopy at $t=1$, we obtain a diffeomorphism $f_\alpha : M_1 \rightarrow M_1$ satisfying $f_\alpha|_{H_1} = {\rm id}$.
For $[D] \in \pi_0(F_2)$, where $D : D^2 \hookrightarrow M_2$ is a neat embedding, 
the action of $\alpha$ on $[D]$ gives $\alpha \cdot [D] = [f_\alpha \circ D]$. 

For $\hat{\xi}_j$, 
the diffeomorphism $f_{\hat{\xi}_j}$ can be regarded as a self-diffeomorphism of $M$ fixing $\nu(H_0) \cup \S_0(H_1)$.
Consider $\hat{\rho}_\gamma$, where $\gamma \in \hat{\pi}_1(M_1;I_j,S_0)$. 
Denote the meridian of $I_j$ in $M$ as $m_j$.
We can select the meridian $m_j$ such that $m_j \subset M_2 \subset M$.
The diffeomorphism $f_{\hat{\rho}_\gamma}$ is a barbell diffeomorphism $\beta_{(m_j,S_0,\gamma)}:M_2 \rightarrow M_2$, which fix the boundary of $M_2$.

We are now heading for a direct sum decomposition of $\pi_1(F_1)$. 
For any $\alpha \in \pi_1(F_1)$, $f_\alpha$ is a self-diffeomorphism of $M$ fixing $\nu(H_0) \cup \S_0(H_1)$. 
The induced bundle map of $f_\alpha$ on the normal bundle of $I_j$ is identical on the boundary and the sub-bundle tangent to $H_1$.
Hence $f_\alpha$ gives an elements $n_j$ of $\pi_1(SO(2)) \cong \ZZ$.
Now we obtain a homomorphism
\begin{equation}\label{mcR} 
  \mathcal{R} : \pi_1(F_1) \rightarrow \ZZ^{2g}, \quad \alpha \mapsto (n_j)_{1 \leq j \leq 2g}.
\end{equation}
By the computation in \cite[Lemma 4.10]{LWXZ25}, it holds that
$$ \mathcal{R}(\hat{\xi}_k) = (2\delta_{jk})_{1 \leq j \leq 2g}, \quad \mathcal{R}(\hat{\rho}_\gamma) = 0.$$
Therefore, the map $\frac{1}{2}\mathcal{R}: \pi_1(F_1) \rightarrow \ZZ^{2g}$ is a splitting map for (\ref{ses1}).

\subsection{The Dax isomorphisms on $M_2$}\label{subsec:daxofm2}
In this subsection, we study the Dax isomorphism on $M_2$. 
We denote the fundamental group of $M$ as $\pi$.
Note that the inclusion maps of $M_1$ and $M_2$ induce isomorphisms of their fundamental groups and $\pi$.

Fix an embedded arc $i: I \hookrightarrow H_2$,
whose image lies in a small collar neighbourhood of $\partial H_2$.
Let $I_0 = \S_d \circ i$.
Then there is an abelian group homomorphism
\begin{equation}
  \F : \pi_1({\rm Emb}_\partial(I,M_2);I_0) \rightarrow \pi_2(M_2;I_0),
\end{equation}
which is just the tracing map defined in (\ref{tracemap}).
Let $\pi_1^D({\rm Emb}_\partial(I,M_2);I_0) = \ker \F$. 

Consider the map $d_3: \pi_3(M_2;I_0) \rightarrow \ZZ[\pi \setminus 1]$.
Since $M_2$ is homotopy equivalent to $\Sigma \vee S^2$, 
$\pi_3(M_2)$ is generated by the Whitehead product of elements in $\pi_2(M_2)$.
Recall that as Definition \ref{S0}, $S_0: S^2 \hookrightarrow M_2$ is the embedding of the fiber over $p \in H_2 \subset \Sigma$.
Without loss of generality, we may assume $S_0$ maps the basepoint of $S^2$ to the image of $I_0$.
Denote the pointed homotopy class of $S_0$ as $[S_0] \in \pi_2(M_2;I_0)$.
We have the isomorphisms
$$\pi_2(M_2) \cong \pi_2(\tilde{M}_2) \cong H_2(\tilde{M}_2) \cong \ZZ[\pi],$$
where $\tilde{M}_2$ is the universal covering of $M_2$.
Namely, $\pi_2(M_2)$ has a generating set $\{\alpha \cdot [S_0]\}_{\alpha \in \pi}$,
where the dot $\cdot$ refers to the action of $\pi \cong \pi_1(M_2)$ on $\pi_2(M_2)$.
Hence the equivariant intersection form $\lambda : \pi_2(M_2;I_0) \times \pi_2(M_2;I_0) \rightarrow \ZZ[\pi]$ of $M_2$ is trivial.
Due to \cite[Proposition 3.14]{KT24}, we obtain that
$$ d_3([a,b]_{\rm Wh}) = \lambda(a,b) + \lambda(b,a) = 0.$$
Therefore, the Dax isomorphism of $M_2$ is
$$ \Dax_{M_2}: \pi_1^D({\rm Emb}_\partial(I,M_2);I_0) \rightarrow \ZZ[\pi \setminus 1].$$

\subsection{Identification of $\pi_2(M_2,\partial D_0)$}\label{idofpi2}

Recall that, as (\ref{scmap}), we have the scanning map
$$ \langle - , - \rangle : \pi_0(F_2) \times \pi_0(F_2) \rightarrow \pi_1({\rm Emb}_\partial (I,M_2);I_0). $$
For $\alpha,\beta$ in $\pi_0(F_2)$, it is a direct try to consider the relative Dax invariants ${\rm Dax}_{M_2} (\langle \alpha , \beta \rangle)$.
However, there is a premise that $\langle \alpha , \beta \rangle$ should be an element of $\pi_1^D({\rm Emb}_\partial (I,M_2);I_0)$.
In this subsection, we show that the premise, up to a priori choice of actions of $\pi_1(F_1)$, can always be achieved.

\begin{defn}
  Denote the image of the embedding $\S_0|_{H_2}: D^2 \hookrightarrow M_2$ as $D_0$.
  Let $\pi_2(M_2,\partial D_0)$ be the set consisting of the maps $D^2 \rightarrow M_2$ coincide with $\S_0|_{H_2}$ on the boundary,
  up to homotopy relative the boundaries.
\end{defn}

Given $\alpha = \sum_{j=1}^{m} \epsilon_j g_j \in \ZZ[\pi]$, where $\epsilon_j \in \{1,-1\}$ and $g_j \in \pi$.
Since the normal bundle of $S_0$ in $M_2$ is trivial, we can pick $m$ copies $S_1, \cdots, S_m$ of $S_0$.
Note that for any $1 \leq j \leq m$, $D_0 \cup S_k$ is simply connected.
So a path with endpoints in $D_0 \cup S_k$ represents an element in $\pi$.
We can pick an embedded arc $\gamma_j: I \hookrightarrow M_2$,
satisfying $\gamma_j(0) \in {\rm Int}(D_0)$, $\gamma_j(1) \in S_j$,
and that $\gamma_j$ represents $g_j \in \pi$. 
We tube $\S_0$ along $\gamma_j$ to $S_k$ (respectively, the orientation reversing of $S_k$)
if $\epsilon_j = 1$ (respectively, $\epsilon_j = -1$), for $1 \leq j \leq m$.
Denote the resulting map as $i_{\alpha}$.
Then there is a map
$$\eta: \ZZ[\pi] \rightarrow \pi_2(M_2,\partial D_0), \quad \alpha \mapsto [i_\alpha].$$
The map $\eta$ has the inverse 
$$\lambda_0 = \lambda(-,\S_0|_{H_2}): \pi_2(M_2,\partial D_0) \rightarrow \ZZ[\pi].$$
Here $\lambda$ refers to the equivariant intersection form defined as \cite[Subsection 3.4]{LWXZ25} for pair of elements in $\pi_2(M_2,\partial D_0)$.

Recall that in Definition \ref{S0}, we define the element $\hat{\rho}_\gamma \in \pi_1({\rm Emb}'_\partial (\sqcup_{2g} I,M_1)).$
Its image under the arc pushing is $\beta_{(m_j,S_0,\gamma)}$.
We denote it as $\beta_\gamma$ for convenience.
The meridian $m_j$ intersect $D_0$ at two points.
We denote the two points as $x_+,x_-$ up to the sign of the intersection.
We isotope the starting point of $\gamma$ in $m_k$ to $x_+$ (respectively, $x_-$) to obtain a path representing an element $g_+$ (respectively, $g_-$) of $\pi$. 
As \cite[Lemma 3.30]{LWXZ25},
the action of barbell diffeomorphisms on $\pi_2(M_2,\partial D_0)$ can be described by the following lemma.

\begin{lem}\label{classchange}
  For any $[i] \in \pi_2(M_2,\partial D_0)$ and any barbell diffeomorphism of the form $\beta_\gamma$, it holds that
  \begin{equation}\label{isotopeh1}
    \lambda_0([\beta_\gamma \circ i]) - \lambda_0([i]) = g_+ - g_- - g_+^{-1} + g_-^{-1}.
  \end{equation}
\end{lem}

The following lemma shows that, 
up to suitable actions of barbell diffeomorphisms, 
we can always assume that $[i] = [\S_d|_{H_2}] \in \pi_2(M_2,\partial D_0)$ for some $d$.

\begin{lem}\label{moveh1}
  Let $i : D^2 \rightarrow M_2$ be an embedding that equals to $\S_0$ near $\partial D^2$.
  Suppose the algebraic intersection number of $i(D^2)$ and $D_0$ are $d$.
  Then there exists a sequence of elements $\{[\gamma_k] \in \hat{\pi}_1(M_1;I_{i_k},S_0)\}_{1 \leq k \leq n}$, 
  such that
  $$ [\beta_{\gamma_n}^{\pm 1} \circ \cdots \circ \beta_{\gamma_1}^{\pm 1} \circ i] = [\S_d|_{H_2}] \in \pi_2(M_2;\partial D_0).$$
\end{lem}

The proof of Lemma \ref{moveh1} is the same as \cite[Proposition 3.31]{LWXZ25}.
We give a sketch of proof here.
\begin{proof}[Sketch of the proof.]
  We have given an identification of $\pi_2(M_2;\partial D_0)$ to $\ZZ[\pi]$.
  By Lemma \ref{classchange}, 
  we need to find a sequence $\{[\gamma_k] \in \hat{\pi}_1(M_1;I_{i_k},S_0)\}_{1 \leq k \leq n}$,
  such that the associated pairs $\{(g_{k,+},g_{k,-})\}_{1\leq k \leq n}$ in $\pi$ satisfy
  $$ \sum_{k=1}^{n} \epsilon_k(g_{k,+} - g_{k,-} - g_{k,+}^{-1} + g_{k,-}^{-1}) = \lambda_0([i]) - \lambda_0([\S_d|_{H_2}]) \in \ZZ[\pi], $$
  for some $\epsilon_k \in \{1,-1\}$.
  Let  $\mu: \pi_2(M_2;\partial D_0) \rightarrow \ZZ[\pi \setminus \{1\}] / (g \sim g^{-1})$ be the equivariant self-intersection form.
  Since $i$ and $\S_d|_{H_2}$ are embeddings, 
  it holds that $\mu([i]) = \mu([\S_d|_{H_2}]) = 0$.
  By the construction of $\eta$ and $\mu$,
  we have that $\mu \circ \eta(1) = 0$, $\mu \circ \eta(g) = [g]$.
  Since we have
  $$ \mu \circ \eta( \lambda_0([i]) - \lambda_0([\S_d|_{H_2}])) = \mu([i]) - \mu([\S_d|_{H_2}]) = 0,$$
  it deduces that $\lambda_0([i]) - \lambda_0([\S_d|_{H_2}])$ is generated by elements of the form $h - h^{-1}$ for $h \in \pi \setminus \{1\}$ and $1 \in \pi$.
  
  There is an abelian group homomorphism $\varepsilon: \ZZ[\pi] \rightarrow \ZZ$, 
  mapping all the $g \in \pi$ to $1 \in \ZZ$.
  The composition $\varepsilon \circ \lambda$ is the algebraic intersection form.
  Since we have 
  $$ \varepsilon(\lambda_0([i]) - \lambda_0([\S_d|_{H_2}])) = d - d = 0, $$
  the element $\lambda_0([i]) - \lambda_0([\S_d|_{H_2}])$ can only be 
  generated by elements of the form $h - h^{-1}$ for $h \in \pi \setminus \{1\}$.
  
  For any $h \in \pi \setminus \{1\}$,
  we can find a sequence $1 = h_0, \cdots, h_m = h$ of elements in $\pi$,
  satisfying the following.
  \begin{itemize}
    \item[(1)]For any $1 \leq j \leq m$, the pair $(h_j,h_{j-1})$ equals to $(g_+,g_-)$ for some $\gamma$.
    \item[(2)]It holds that $$ h - h^{-1} = \sum_{j=1}^{m} \pm (h_j - h_{j-1} - h_j^{-1} + h_{j-1}^{-1}).$$
  \end{itemize}
  It concludes our proof.
\end{proof}

The following definition gives a subgroup of $\pi_1(F_1)$, 
gathering all different choices of the elements in $\pi_1(F_1)$, 
such that their image under the arc pushing satisfy the condition (\ref{isotopeh1}).

\begin{defn}\label{pqk}
  Pick $[i] \in F_2$.
  \begin{itemize}
    \item  There is a homomorphism
    $$ \mathcal{P} : \pi_1(F_1) \rightarrow \ZZ[\pi], \quad \alpha \mapsto \lambda_0[i] - \lambda_0[f_\alpha \circ i]. $$
    \item Let $K$ be the subgroup $\ker \mathcal{P}$ of $\pi_1(F_1)$.
    Recall that $\mathcal{R}$ is the homomorphism defined in (\ref{mcR}).
    Let $K_0$ be $K \cap \ker{\mathcal{R}}$.
    \item There is a homomorphism
    $$ \mathcal{Q} : K \rightarrow \ZZ[\pi \setminus \{1\}]^\sigma, \quad \alpha \mapsto \Dax_{M_2}(i,f_\alpha \circ i). $$ 
  \end{itemize}
\end{defn}

\begin{rem}
  It is proved in \cite[Corollary 3.32, Lemma 4.1]{LWXZ25} that
  the definitions of $\mathcal{P}$ and $\mathcal{Q}$ are independent on the choices of $i$ in $F_2$. 
  Hence $\mathcal{P}$ and $\mathcal{Q}$ are well-defined homomorphisms.
  We can let $i$ be $D_0$ for convenience in Definition \ref{pqk}.
\end{rem}

The following lemma gives a reduction when showing the well-definedness of the Dax invariants in Subsection \ref{ssec:bbcal}.

\begin{lem}\label{K0}
  For any $\alpha \in K$, there is an element $\alpha' \in K$ satisfying $\alpha - \alpha' \in K_0$ and $\mathcal{Q}(\alpha') = 0$.
\end{lem}

\begin{proof}
  Assume that $\mathcal{R}(\alpha) = (2a_1,\cdots,2a_{2g}) \in \ZZ^{2g}.$
  Denote the 1-skeleton of the fixed cellular structure of $\Sigma$ as ${\rm sk}^1$.
  Since $\ZZ^{2g} = H^1(\Sigma;\ZZ) = H^1({\rm sk}^1;\ZZ) = [{\rm sk}^1,SO(2)]$,
  we can pick $f:{\rm sk}^1 \rightarrow SO(2)$ corresponding to $\mathcal{R}(\alpha)$.
  Define an inclusion 
  $$i : SO(2) \rightarrow SO(3), \quad P \mapsto {\rm diag}\{P,1\}.$$
  Denote $h_1 = i \circ f : {\rm sk}^1 \rightarrow SO(3)$. 
  If we restrict $h_1$ on any $S^1$ wedge sum component of ${\rm sk}^1 = \vee_{2g} S^1$, we get a null-homotopic map $S^1 \rightarrow SO(3)$.
  Therefore, we can extend $h_1$ to $\tilde{h}_1 : H_0 \cup H_1 \rightarrow SO(3)$, mapping $H_0$ and $\partial H_1$ to ${\rm id} \in SO(3)$.
  Because $\pi_2(SO(3)) = 0$, there is a homotopy $\tilde{h}_t : H_0 \cup H_1 \rightarrow SO(3) (t \in [0,1])$ 
  between $\tilde{h}_1$ and the constant map to ${\rm id}$, relative to $H_0 \cup \partial H_1$.

  If we restrict the $S^2$-bundle $M \rightarrow \Sigma$ to $H_0 \cup H_1$, we get a trivial $S^2$-bundle.
  Define an ambient isotopy
  $$ r_t : (H_0 \cup H_1) \times S^2 \rightarrow (H_0 \cup H_1) \times S^2, \quad (x,v) \mapsto (x,\tilde{h}_t(x)v). $$
  By the definition of $\tilde{h}_t$, the isotopy $r_t$ fix $(H_0 \cup \partial H_1) \times S^2$ for all $t \in [0,1]$,
  thus we can extend $r_t$ by identity to $\tilde{r}_t: M \rightarrow M$.

  Let $\alpha'$ be the loop of embeddings of $H_1$ defined as 
  $$ \alpha'(t) : H^1 \hookrightarrow M, \quad x \mapsto \tilde{r}_t \circ \S_d(x). $$
  Then we can pick $\tilde{r}_1$ to be a representative of $f_{\alpha'}$.
  Since $\tilde{h}_1$ factors through $SO(2)$, $\tilde{r}_1$ fixes all points of form $(x,(0,0,1)^T)$ for $x \in H_1$,
  hence $f_{\alpha'} \circ \S_0 = \S_0$.
  We obtain that $\mathcal{R}(\alpha') = \mathcal{R}(\alpha)$, $\alpha' \in K$ and $\Dax_{M_2}(D_0,f_{\alpha'}(D_0))=0$, which concludes our proof.
\end{proof}

\subsection{Definition of the Dax invariants}

We give the concrete definition of the relative Dax invariants on $M$ in this subsection.

\begin{defn}
  Let $\mathcal{C}$ be the set consisting of all the conjugacy classes of $\pi$ except the conjugacy class $\{1\}$ of the unit. 
  There is a canonical projection map
  $$ \mathfrak{p} : \ZZ[\pi \setminus \{1\}] \rightarrow \ZZ[\mathcal{C}].$$  
\end{defn}

\begin{prop}\label{0cp}
  The following composition gives zero homomorphism
  $$ K \xrightarrow{\mathcal{Q}} \ZZ[\pi \setminus \{1\}]^\sigma \hookrightarrow \ZZ[\pi \setminus \{1\}] \xrightarrow{\mathfrak{p}} \ZZ[\mathcal{C}].$$
\end{prop}

We leave the proof of Proposition \ref{0cp} to the next subsection.
Now we use the notation in the fibration tower (\ref{fbtw}) to give the definition of the Dax invariants for embedded surfaces.

Given $[i_1],[i_2] \in \pi_0(E_2)$, 
the homotopy of $i_1 |_{H_0 \cup H_1}$ and $i_2 |_{H_0 \cup H_1}$ in the 4-manifold $M$ can be modified to an ambient isotopy.
Thus we may assume that $i_1 |_{H_0 \cup H_1} = i_2 |_{H_0 \cup H_1} = \S_d|_{H_0 \cup H_1}$.
By Proposition \ref{moveh1}, there are $\alpha_1, \alpha_2 \in \pi_1(F_1)$, satisfying $[f_{\alpha_1} \circ i_1] = [f_{\alpha_2} \circ i_2] \in \pi_2(M_2,\partial D_0)$. 
Now the following Dax invariants in $M_2$
$$ \Dax_{M_2}(f_{\alpha_1} \circ i_1, f_{\alpha_2} \circ i_2) \in \ZZ[\pi \setminus \{1\}]^{\sigma} $$
is well-defined.

\begin{defn}\label{defofdax}
For a pair of pointed embedded surfaces $i_1,i_2: \Sigma \hookrightarrow M$, the relative Dax invariant is defined by
$$ \Dax(i_1,i_2) = \mathfrak{p} \circ \Dax_{M_2}(f_{\alpha_1} \circ i_1, f_{\alpha_2} \circ i_2) \in \ZZ[\mathcal{C}].$$
\end{defn}

\subsection{Barbell calculus and the proof of Proposition \ref{0cp}}\label{ssec:bbcal}

Due to Lemma \ref{K0}, we know that $\mathcal{Q}(K) = \mathcal{Q}(K_0) \subset \ZZ[\pi\ \{1\}]^{\sigma}$. 
Hence we can reduce Proposition \ref{0cp} to the following simpler one.

\begin{prop}\label{red0cp}
  The following composition gives zero homomorphism
  $$ K_0 \xrightarrow{\mathcal{Q}} \ZZ[\pi \setminus \{1\}]^\sigma \hookrightarrow \ZZ[\pi \setminus \{1\}] \xrightarrow{\mathfrak{p}} \ZZ[\mathcal{C}].$$
\end{prop}

The proof of Proposition \ref{red0cp} is an application of the "barbell calculus" technique in \cite[Subsection 4.3]{LWXZ25}, 
and a slight modification of the proof of \cite[Proposition 4.3]{LWXZ25}.
Here we give a brief sketch of the proof, and emphasize the reasons why the original proof still applies.


For elements in $\pi_1({\rm Emb}^{\rm Fr}_\partial (\sqcup_{2g} I,M_1))$ of the form $\tilde{\rho}_\gamma$, 
there results after arc-pushing are special barbell diffeomorphisms called the vertical-meridian barbell diffeomorphisms.

The defining data for a vertical-meridian barbell diffeomorphism consists of
\begin{itemize}
  \item a $1$-handle $h$ of $\Sigma$, the center $q$ of $h$ and an orientation $\mathfrak{o}$ for the meridian $m$ of $\S_0(h)$ at $\S_0(q)$;
  \item the point $p$ in the $2$-handle $H_2$ of $\Sigma$, whose $S^2$-fiber under the projection $M \rightarrow \Sigma$ is $S_0$;
  \item a smooth path $\gamma$ in $\Sigma$ from $p$ to $q$, required that ${\rm int}(\gamma)$ does not intersect $p$ and the core of $h$.
\end{itemize}
We may lift $\gamma$ to an embedded path in $M$, and perturb it at endpoints to a path $\hat{\gamma}$ connecting $m$ and $S_0$.
Then the barbell diffeomorphism given by $(S_0,m,\hat{\gamma})$ is a vertical-meridian barbell diffeomorphism, 
denoted as $\beta_{p,q,\gamma,\mathfrak{o}}$.
We remark that the orientation of $S_0$ is already chosen a priori to be the orientation of the $S^2$-fibers.

Consider the normal bundle $\nu$ for $c \rightarrow \S_0(h)$, where $c$ is the core of $h$.
For convenience, we fix a metric on $\Sigma$, inducing metrics on $\S_0(\Sigma)$ and $\nu$.
Then the orientation $\mathfrak{o}$ of $m$ corresponds to a choice of a unit section $s$ of $\nu$.
We push $c$ along $\S_0(\Sigma)$ slightly in the direction $s(c)$ (respectively, $-s(c)$), 
to get a path $\gamma_+$ (respectively, $\gamma_-$), satisfying that $\gamma_+(1)$ and $\gamma_-(1)$ are in $\S_0(H_2)$.
The concatenation of the paths $\gamma$ and $\gamma_+$ (respectively, $\gamma_-$) 
represents an element $g_+$ (respectively, $g_-$) of the $\pi_1(M,\S_0(H_2))$,
which is canonically isomorphic to $\pi$.
For a pair $(g,g')$ of elements in $\pi$, 
we say the pair is adjacent, if $g',g$ differ by the right product of a standard generator of $\pi$.
Note that the pair $(g_+,g_-)$ is adjacent.
We call $(g_+,g_-)$ the adjacent pair for $\beta_{p,q,\gamma,\mathfrak{o}}$.
According to Lemma \ref{classchange}, it holds that
$$\lambda_0([\beta_{p,q,\gamma,\mathfrak{o}} \circ D_0]) - \lambda_0([D_0]) = g_+ - g_- - g_+^{-1} + g_-^{-1}. $$



The following proposition shows that, 
the image of the relative Dax invariant under the projection $\mathfrak{p}$ is unchanged,
when we homotope the defining data $\gamma$ of a vertical-meridian barbell diffeomorphism.

\begin{prop}\label{homotopepath}
  Let $(p,q,\gamma,\mathfrak{o}), (p,q,\gamma',\mathfrak{o})$ be defining data for the vertical-meridian barbell diffeomorphisms $\beta$ and $\beta'$, respectively. 
  Suppose $\gamma, \gamma'$ are homotopic relative to endpoints.
  Then it holds the following.
  \begin{itemize}
    \item[(1)] $[\beta(D_0)] = [\beta'(D_0)] \in \pi_2(M_2,\partial D_0)$;
    \item[(2)] $\mathfrak{p} \circ \Dax (\beta(D_0), \beta'(D_0)) = 0$.
  \end{itemize}
\end{prop}

The proof of Proposition \ref{homotopepath} based on a careful computation 
when the trajectories of ${\rm int}(\gamma)$ in $M$ go through the endpoints.
Since the base surfaces of the $S^2$-bundles $M$ and $\Sigma \times S^2$ are the same, 
such computation in \cite{LWXZ25} is applicable in our case.
Now we can reduce Proposition \ref{red0cp} to the following proposition.
\begin{prop}\label{existenceofbarbell}
  Let $\{(g_{j,+},g_{j,-})\}_{1 \leq j \leq m}$ be a sequence of adjacent pairs in $\pi$.
  Suppose the sequence $\{(g_{j,+},g_{j,-})\}_{1 \leq j \leq m}$ satisfies that
  \begin{equation}\label{Kcondition}
    \sum_{j=1}^{m} (g_{j,+} - g_{j,-} - g_{j,+}^{-1} + g_{j,-}^{-1}) = 0 \in \ZZ[\pi].
  \end{equation}
  Then there is a sequence $\{ \beta_j \}_{1 \leq j \leq m}$ of vertical-meridian barbell diffeomorphisms, such that
  \begin{equation*}
    \mathfrak{p} \circ \Dax (D_0, \sum_{j=1}^{m} \beta_j(D_0)) = 0.
  \end{equation*}
  Here the sum refers to the composition in $\MCG(M_2)$.
\end{prop}

Proof of Proposition \ref{existenceofbarbell} is based a clarification on "the admissibility of sequences of adjacent pairs".
This argument only involves computation on the universal covering $\tilde{\Sigma}$ of the surface $\Sigma$, 
so it still can be applied to the case we consider.

\subsection{Properties of the Dax invariants}

In this subsection we show some properties of the relative Dax invariants for the embeddings of surfaces in $M$.

\subsubsection{Independency}
The following results can be shown by the same arguments as the proofs for \cite[Lemma 4.5, Lemma 4.6, Lemma 3.1]{LWXZ25}. 

\begin{lem}\label{hddecompositon}
  The Dax invariants are independent on the choices of handle decompositions satisfying that the 0-handle containing the basepoint.
\end{lem}

\begin{lem}
  The Dax invariants are independent on the choices of the basepoint of $\Sigma$.
\end{lem}

Let $\mathcal{E}'_d$ be the set consisting of the non-parameterized embedded surfaces in $M$,
satisfying that there exist parameterizations for them homotopic to $\S_d$.
There is an equivalence relation $\sim'$ on $\mathcal{E}'_d$.
For $\Sigma_1$ and $\Sigma_2$ in $\mathcal{E}'_d$,
we say $\Sigma_1 \sim' \Sigma_2$,
if there are parameterizations $i_1,i_2$ for $\Sigma_1,\Sigma_2$ respectively,
such that $i_1$ and $i_2$ are isotopic in $M$.

\begin{lem}\label{parameter}
  The map
  $$ \pi_0 ({\rm Emb}^{[\S_d]}_\bullet (\Sigma,M)) \ \rightarrow \ \mathcal{E}'_d / \sim' $$
  mapping the connected component of ${\rm Emb}^{[\S_d]}_\bullet (\Sigma,M)$ containing an embedding $i$ to the equivalence class of the image of $i$, 
  is a bijection.
  Therefore, the Dax invariants can be regarded as being defined for pairs of elements in $\mathcal{E}'_d / \sim'$.
\end{lem} 

\subsubsection{Additivity}

The additivity of the Dax invariants for embedded disks immediately imply the following lemma. 

\begin{lem}
  For $i_1,i_2,i_3 \in {\rm Emb}^{[\S_d]}_\bullet (\Sigma,M)$, it holds the following
  $$ \Dax(i_1,i_3) = \Dax(i_1,i_2) + \Dax(i_2,i_3). $$
\end{lem}

\subsubsection{Naturality}

Regarding the naturality of the Dax invariants, we have the following lemma.

\begin{lem}\label{naturality}
  For any 
  pointed diffeomorphism $f: M \rightarrow M$, it holds that
  $$ \Dax(f \circ i_1,f \circ i_2) = f_* \circ \Dax(i_1, i_2). $$
  Here $f_*$ is the induced automorphism on $\ZZ[\mathcal{C}]$, and $i_1, i_2: \Sigma \hookrightarrow M$ are embeddings that homotopic to $\S_d$.
\end{lem}

\begin{proof}
  Consider the composed map $\Sigma \xrightarrow{\S_0} M \xrightarrow{f} M \xrightarrow{{\rm pr}} \Sigma$.
  By the Dehn-Nielson theorem, there is a pointed self-diffeomorphism $f_{\dagger}$ of $\Sigma$, 
  unique up to isotopy, inducing the same automorphism on $\pi$. 
  The pull-back $S^2$-bundle $(f_\dagger)^* M$ is isomorphic to $M$.
  We fix an bundle isomorphism $M \rightarrow (f_\dagger)^* M$.
  Then we compose such fixed isomorphism with the bundle map $(f_\dagger)^* M \rightarrow M$, 
  and name the results by $F: M \rightarrow M$.
  Note that $M$ is the sphere bundle of the vector bundle $\xi \oplus \underline{\RR}$ over $\Sigma$,
  where $\xi$ is the real vector bundle with $w_1(\xi) = 0, w_2(\xi) \ne 0$,
  and $\underline{\RR}$ is the trivial real line bundle.
  Thus we may require that the bundle isomorphism $M \rightarrow (f_\dagger)^* M$ fixes the $\underline{\RR}$ component.
  Namely, we may require that $F \circ \S_0 = \S_0 \circ f_\dagger$.
  
  Firstly, we show that
  $$ \Dax(F \circ i_1, F \circ i_2) = F_* \circ \Dax(i_1,i_2), $$
  where $F_*$ is the induced automorphism on $\ZZ[\mathcal{C}]$.
  Without loss of generality, after a first isotopy, we may assume that $i_1|_{H_0 \cup H_1} = i_2|_{H_0 \cup H_1} = \S_d|_{H_0 \cup H_1}$, 
  and $i_1|_{H_2}, i_2|_{H_2}: D^2 \hookrightarrow M_2$ are homotopic relative to $\partial D^2$.
  Note that the unparameterized embedded surfaces $F \circ i_1(\Sigma)$ and $F \circ i_2(\Sigma)$ 
  coincide on the union $H'$ of a 0-handle and $2g$ 1-handles.
  To be precise, $H'= F \circ \S_d(H_0 \cup H_1) = F \circ \S_0(H_0 \cup H_1) \subseteq \S_0(\Sigma)$. 
  Recall that $M_2$ is defined as Definition \ref{m1m2}.
  Let $M_2'$ be the image of $F|_{M_2}$.
  There is a commutative diagram for fundamental groups and induced homomorphisms between them
  \[
  \begin{tikzcd}
    \pi_1(M_2) \arrow[r,"F_\#"] \arrow[d,"{\rm inclusion}_\#"] & \pi_1(M'_2) \arrow[d,"{\rm inclusion}_\#"] \\
    \pi_1(M) \arrow[r,"F_\#"] & \pi_2(M).
  \end{tikzcd}
  \]
  Note that the two vertical arrows are isomorphisms.
  Going through the definition for the Dax invariants, we have
  \begin{equation*}
    \begin{split}
      \Dax(F \circ i_1, F \circ i_2) &= \Dax (F \circ i_1(\Sigma), F \circ i_2(\Sigma)) \\ 
      &= \mathfrak{p} \circ \Dax_{M_2'} (F \circ i_1|_{H_2}(D^2), F \circ i_2|_{H_2}(D^2)) \\
      &= \mathfrak{p} \circ F_\# \circ \Dax_{M_2} (i_1|_{H_2}(D^2), i_1|_{H_2}(D^2)) \\
      &= F_* \circ \Dax (i_1, i_2).
    \end{split}
  \end{equation*}
  The first equation holds because of Lemma \ref{parameter}.
  The second and fourth equations hold because of Lemma \ref{hddecompositon}.
  The third equation holds by checking the definition of the Dax invariants for embedded disks.

  Secondly, by the same argument as \cite[Lemma 4.8]{LWXZ25}
  we know that if a pointed self-diffeomorphism g of $M$ induces the identity map on $\pi$, we have
  $$ \Dax(g \circ i_1, g \circ i_2) = \Dax(i_1,i_2). $$
  Now we can heading for our results. We have
  \begin{equation*}
    \begin{split}
      \Dax(f \circ i_1, f \circ i_2) &= \Dax(F \circ F^{-1} \circ f \circ i_1, F \circ F^{-1} \circ f \circ i_2) \\
      &= F_* \circ \Dax(F^{-1} \circ f \circ i_1, F^{-1} \circ f \circ i_2) \\
      &= F_* \circ \Dax(i_1,i_2) \\
      &= f_* \circ \Dax(i_1,i_2).
    \end{split}
  \end{equation*}
  It concludes our proof.
\end{proof}

\section{Isotopy classification of surfaces with a common dual sphere}\label{sec:isoclass}

Recall that $b=\S_0(b_0)$ is the basepoint of $M$, and $S'_0$ is the fiber of the $S^2$-bundle $p : M \rightarrow \Sigma$ at $b_0$.
In this section, we are going to show the following theorem.

\begin{thm}\label{isotopyclass}
  Suppose $\Sigma_1, \Sigma_2$ are two unparameterized smoothly embedded surfaces in $M$.
  Suppose the following conditions hold.
  \begin{itemize}
    \item[(1)] The embedded surfaces $\Sigma_1, \Sigma_2$ intersect $S'_0$ transversely and positively at the point $b$.
    \item[(2)] The algebraic intersection numbers of $\Sigma_1, \Sigma_2$ with $\S_{-1}$ are identical to $d$.
    \item[(3)] The embedded surfaces $\Sigma_1, \Sigma_2$ agree with $\S_d(\Sigma)$ on a neighbourhood of $S'_0$.
  \end{itemize} 
  Then there are parameterizations $i_1, i_2$ for $\Sigma_1, \Sigma_2$, respectively, satisfying the following conditions.
  \begin{itemize}
    \item[(1)] There is a neighbourhood $\nu$ of $S'_0$, such that the intersection of $\Sigma_1,\Sigma_2$ and $\S_d(\Sigma)$ with $\nu$ is the same 2-disk.
    \item[(2)] In $M \setminus \nu$, the restriction of $i_1, i_2$ and $\S_d$ to the complement of the preimage of $\nu$, are homotopic relative to the boundary.
    \item[(3)] In $M$, $\Sigma_1$ and $\Sigma_2$ are isotopic if and only if $\Dax(\Sigma_1,\Sigma_2) = 0 \in \ZZ[\mathcal{C}]^{\sigma}$. 
  \end{itemize}
\end{thm}

We need the following lemma proved in \cite[Lemma 6.4]{LWXZ25}.

\begin{lem}\label{minusd}
  Let $D$ be a smoothly embedded 2-disk in $\Sigma$.
  Suppose $h: \Sigma \setminus D \rightarrow \Sigma \setminus D$ is a continuous map satisfying $h|_{\partial (\Sigma \setminus D)} = {\rm id}$.
  Then $h$ is homotopic to a homeomorphism relative to $\partial (\Sigma \setminus D)$.
\end{lem}

\begin{proof}[Proof of Theorem \ref{isotopyclass} (1),(2)]
  By the assumptions, we can choose the parameterization $i_1, i_2$ to satisfy $i_1(b_0) = i_2(b_0) = \S_d(b_0) = b$.
  Let $i$ be $i_1$ or $i_2$. 
  Here let $\MCG^+_{\partial}(D^2)$ be the mapping class group of 2-disk, 
  consisting of the orientation-preserving self-homeomorphisms being identity on the boundary up to isotopy relative to the boundary.
  Since $\MCG^+_{\partial}(D^2)$ is trivial, after an isotopy of the parameterization, 
  we may assume that there is a smoothly embedded 2-disk $D_0$ containing $b_0$ in $\Sigma$,
  satisfying $i|_{D_0} = \S_d|_{D_0}$.
  Because $i$ and $\S_d$ only intersect their common geometric dual $S'_0$ at $b$, 
  there is a 2-disk $D \subset D_0$ containing $b_0$,
  satisfying that $i(\Sigma) \cap p^{-1}(D), \S_d(\Sigma) \cap p^{-1}(D)$ are identical to the same 2-disk $i(D)$.
  Denote $p^{-1}(D)$ as $\nu$.
  It concludes the proof for Theorem \ref{isotopyclass}(1).

  Note that $M \setminus \nu$ is diffeomorphic to $(\Sigma \setminus D) \times S^2$.
  In the following, we fix a diffeomorphism $\phi: M \setminus \nu \rightarrow (\Sigma \setminus D) \times S^2$.
  Let $p_1, p_2$ be the projection of $M \setminus \nu$ to $\Sigma \setminus D$ and $S^2$, respectively.
  It is sufficient to show that $p_j \circ i|_{\Sigma \setminus D}$ is homotopic to $p_j \circ \S_d|_{\Sigma \setminus D}$ 
  relative to $\partial(\Sigma \setminus D)$ for $j=1,2$.
  Consider the composition map
  $$ \Sigma \setminus D \overset{i}{\hookrightarrow} M \setminus \nu \xrightarrow{p_1} \Sigma \setminus D. $$
  It holds that $p_1 \circ i|_{\partial (\Sigma \setminus D)} = p_1 \circ \S_d|_{\partial (\Sigma \setminus D)} = {\rm id}$.
  According to Lemma \ref{minusd}, the composition map is homotopic to a homeomorphism relative to the boundary.
  Thus, after a reparameterization, we can require that $p_1 \circ i|_{\Sigma \setminus D}$ is homotopic to the identity.
  By the construction of $\S_d$, we have $p_1 \circ \S_d|_{\Sigma \setminus D}$ is homotopic to the identity.
  Consider the composition map
  $$ \Sigma \setminus D \overset{i}{\hookrightarrow} M \setminus \nu \xrightarrow{p_2} S^2. $$
  We can pick the diffeomorphism $\phi: M \setminus \nu \rightarrow (\Sigma \setminus D) \times S^2$
  satisfying that $p_2 \circ \S_0(\Sigma \setminus D) = \{ (0,0,1)^T \}$.
  Since $p_2 \circ i$ maps $\partial(\Sigma \setminus D)$ to a single point,
  it induces a quotient map $g : \Sigma \rightarrow S^2$.
  The degree of the quotient map $g$ determines the homotopy class of $p_2 \circ i|_{\Sigma \setminus D}$ relative to the boundary. 
  Recall that $\tau \circ \S_0$ is an embeddings that maps $x \in \Sigma$ to $(x,(0,0,-1)^T) \in M$.
  Now we have
  \begin{equation*}
      \begin{split}
        \deg (g) &= (p_2 \circ i|_{\Sigma \setminus D}) \cap_{S^2} (0,0,-1)^{T}\\
        &= (i|_{\Sigma \setminus D}) \cap_{M \setminus \nu} (\tau \circ \S_0|_{\Sigma \setminus D}) \\
        &= i \cap_M (\tau \circ \S_0) \\
        &= d.
    \end{split}
  \end{equation*}
  Here we use the notation $\cap_{S^2}, \cap_{M \setminus \nu}, \cap_M$ to 
  refer to the algebraic intersection numbers in $S^2, M \setminus \nu, M$ respectively.
  The second equation holds because $p_2^{-1}((0,0,-1)^{T}) = \tau \circ \S_0(\Sigma \setminus D)$.
  The third equation holds because $i$ and $\tau \circ \S_0$ do not intersect in $\nu$.
  The fourth equation holds by the given assumptions.
  By a similar argument on the composition map $p_2 \circ \S_d|_{\Sigma \setminus D}$,
  we obtain that $p_2 \circ i|_{\Sigma \setminus D}$ is homotopic to $p_2 \circ \S_d|_{\Sigma \setminus D}$ relative to the boundary.
  It concludes the proof for Theorem \ref{isotopyclass} (2).
\end{proof}

Without loss of generality, we can assume that $\nu = p^{-1} (D)$ contains the neighbourhood of $\S_0(H_0)$ in Definition \ref{m1m2}.
Let $M_1' = M \setminus \nu, M_2' = M_2 \setminus \nu$.
Note that the manifolds $M_1', M_2'$ constructed here is diffeomorphic to those constructed in \cite[Page 53]{LWXZ25}.
Thus, the remaining proof of Theorem \ref{isotopyclass} is an analogy of the proof of \cite[Theorem 6.1(3)]{LWXZ25}.
We give a sketch of the argument.

\begin{proof}[Proof of Theorem \ref{isotopyclass}(3)]
  After smoothing corners, we may regard $M_1'$ and $M_2'$ as manifolds with boundaries.
  They are also codimension-0 submanifolds of $M$.
  Note that $M_1'$ is diffeomorphic to $(\Sigma \setminus D) \times S^2$, $M_2'$ is homotopy equivalent to $ (\vee_{2g} S^1) \vee S^2$.
  They have isomorphic fundamental groups, denote them as $\pi'$.
  Further more, the inclusion maps of $M_1', M_2'$ into $M$ induce surjective homomorphisms from $\pi'$ to $\pi$.

  Let $i$ be $i_1$ or $i_2$.
  According to Theorem \ref{isotopyclass}(1), we have a neat embedding $i|_{\Sigma \setminus D}: \Sigma \setminus D \rightarrow M_1' $.
  By Theorem \ref{isotopyclass}(2) and the fact that homotopy of 1-handles in 4-manifolds can be perturbed to isotopy,
  we obtain an ambient isotopy $P_t (t \in [0,1])$ of $M$ supported in $M_1'$.
  The isotopy $P_t$ satisfies that $P_0 = {\rm id}$, $P \circ i |_{H_0 \cup H_1} = \S_d |_{H_0 \cup H_1}$, 
  and $P_1 \circ i|_{H_2}$ is a neat embedding of 2-disk in $M_2'$.

  Since $P_1 \circ i|_{H_2}$ and $\S_d|_{H_2}$ may not homotopic relative to the boundary,
  we need a further ambient isotopy extended by isotopy of 1-handles to ensure such two embedded disks are homotopic,
  as in Subsection \ref{idofpi2}.
  Recall that $D_0 = \S_0(H_2)$ is a neatly embedded disk in $M_2'$.
  Let $\pi_2(M_2';\partial D_0)$ be the set of maps of disks $M_2'$ identical to $\S_0|_{H_2}$ on the boundary, up to homotopy relative to the boundary.
  Then we have a bijection
  $$ \lambda_0' := \lambda'(-,\S_0|_{H_2}): \pi_2(M_2';\partial D_0) \rightarrow \ZZ[\pi'], $$
  where $\lambda'$ is the equivariant intersection form for disks in $M_2'$.
  An isotopy of $\S_d(H_1)$ to itself can entend any meridian-vertical barbell diffeomorphism on $M_2'$. 
  The action of a meridian-vertical barbell diffeomorphism $\beta$ on $\lambda_0'([i])$ gives
  $$ \lambda_0'([\beta \circ i]) = \lambda_0'[i] + g_+ - g_- - g_+^{-1} + g_-^{-1}, $$
  where $(g_+,g_-)$ is an adjacent pair in $\pi'$.
  By the same argument as the proof of Lemma \ref{moveh1},
  after acting a self-diffeomorphism $\psi$ of $M$ supported in $M_2'$,
  we may further assume that $\psi \circ P_1 \circ i|_{H_2}$ are homotopic to $\S_d|_{H_2}$ relative to the boundary.

  According to the argument above, 
  we may assume that the restrictions of $i_1,i_2$ to $H_0 \cup H_1$ are identical, 
  and the restrictions of $i_1,i_2$ to $H_2$ are homotopic relative to the boundary in $M_2'$.
  According to \cite[Theorem 1.5]{KT24},
  since $i_1|_{H_2}$ has a boundary geometric dual $G$,
  there is an element
  $$ \alpha = \sum_{j=1}^m \epsilon_j g_j \in \ZZ[\pi' \setminus \{1\}]^\sigma, $$
  where $\epsilon_j \in \{1,-1\}$, $g_j \in \pi' \setminus \{1\}$ for all $1 \leq j \leq m$,
  such that $i_1|_{H_2}$ is isotopic to the action of $\alpha$ on $i_2|_{H_2}$ via finger moves and Norman tricks.
  Namely $i_1|_{H_2}$ is isotopic to the embedding $i_2|_{H_2} + {\rm fm}(\alpha)^G$.
  For the concrete expressions of the action of $\ZZ[\pi' \setminus \{1\}]^\sigma$ on the embedded disks, 
  we refer the readers to \cite[Figure 1.6]{KT24}.
  By Definition \ref{defofdax}, the Dax invariant of embeddings $i_1,i_2$ is 
  $$ \Dax(i_1,i_2) = \sum_{j=1}^m \epsilon_j [g_j] \in \ZZ[\mathcal{C}]^\sigma, $$
  where $[g_j]$ is the conjugacy class of the image of $g_j$ under the surjective homomorphism $\pi' \rightarrow \pi$.
  By assumptions, we have $ \Dax(i_1,i_2)=0$.
  Since $\alpha \in \ZZ[\pi' \setminus \{1\}]^\sigma$, there is an element $\beta = \sum_{j=1}^{n} \delta_j h_j \in \ZZ[\pi' \setminus \{1\}]$,
  where $\delta_j \in \{1,-1\}$, such that $\alpha = \beta + \sigma(\beta)$.
  According to \cite[Figure 13,14]{G21}, 
  $i_2|_{H_2} + {\rm fm}(\alpha)^G$ is isotopic to the embedded disk obtained by tubing a self-referential disk on $i_2$ along $\beta$.

  If we isotope the self-referential tube with respect to $h_j$ across $\nu$, we can change $h_j$ to any element with the same image in $\pi$.
  If we isotope the attaching point of the tube with respect to $h_j$ and $i_2(H_2)$, we may change $h_j$ to any element conjugate to it in $\pi$.
  Thus, after an isotopy moving the attaching points and the self-referential tubes, we may assume $\beta = 0$.
  It deduces $\alpha = 0$.
  Hence $i_1$ is isotopic to $i_2$ in $M_2'$ relative to the boundary.
  It concludes the proof of Theorem \ref{isotopyclass}(3).
\end{proof}

Recall that as Definition \ref{deffd}, 
$\mathcal{F}_d$ is the set consisting of the embedded surfaces in $M$ satisfying the three conditions in Theorem \ref{isotopyclass},
and $\approx$ is the isotopy relation on $\mathcal{F}_d$.
The following corollary of Theorem \ref{isotopyclass} is a restatement of Theorem \ref{thm3}.

\begin{cor}\label{cor4.3}
  There is a bijection
 $$ \Lambda : \mathcal{F}_d/\approx  \ \rightarrow \ \ZZ[\mathcal{C}]^\sigma, \quad [\Sigma] \mapsto \Dax(i,\S_d). $$
 Here $i$ is a pointed parameterization for $\Sigma$ such that $i$ is homotopic to $\S_d$ relative to $b_0$.
\end{cor}

\begin{proof}
  For any $\Sigma \in \mathcal{F}_d$,
  such parameterization $i$ of $\Sigma$ exists,
  because of the results (1),(2) of Theorem \ref{isotopyclass}.
  The value of $\Dax(i,\S_d)$ is independent of the choice of $i$ because of Lemma \ref{parameter}.
  For the isotopy invariance of the Dax invariants, the map $\Lambda$ is well-defined.
  The injectivity of $\Lambda$ in deduced from the result (3) of Theorem \ref{isotopyclass}.

  Then we review the construction for a family of barbell diffeomorphisms called the self-referential barbell diffeomorphisms following \cite{LWXZ25} and \cite[Construction 5.3]{G19}.
  Recall that $S_0$ is the fiber of $M$ at the fixed point $p \in H_2 \subset \Sigma$.
  We pick a path $\gamma$ in $\Sigma$ with distinct endpoints $p$ and $q \in {\rm Int}(H_2)$,
  requiring that $(0,1) \cap \gamma^{-1}(q) = \{t_0\}$, 
  and $\gamma$ is self-transverse at $q$. 
  Let $c$ be the element of $\mathcal{C}$ represented by the loop $\gamma|_{[t_0,1]}$.
  Lift $\gamma$ to a path $\tilde{\gamma}$ in $M_2$, 
  such that the only self-intersection of $\tilde{\gamma}$ is $\S_0(q)$.
  Let $m$ be the meridian of the interior of $\tilde{\gamma}$ at $\S_0(q)$.
  We shrink $\tilde{\gamma}$ at the endpoint $\S_0(q)$
  to a path $\hat{\gamma}$ that connecting $S_0$ and $m$.
  The defining data $(S_0,m,\hat{\gamma})$ gives a barbell diffeomorphism $\beta_{(S_0,m,\hat{\gamma})}$.
  We call $\beta_{(S_0,m,\hat{\gamma})}$ a self-referential barbell diffeomorphism,
  and call $c$ the associated class of $\beta_{(S_0,m,\hat{\gamma})}$.

  Now we prove the surjectivity of $\Lambda$.
  Recall that when constructing the embedding $\S_d$, we choose $|d|$ points $p_1,\cdots,p_{|d|}$ in $H_2$.
  For any generator $[g]+[g^{-1}]$ of $\ZZ[\mathcal{C}]^\sigma$,
  there is a self-referential barbell diffeomorphism $\beta$ of $M_2$,
  such that $\beta$ is supported away from the $S^2$-fibers of $M$ over $p_1,\cdots,p_{|d|}$,
  and the associated class of $\beta$ is $[g]$.
  We extend $\beta$ by identity, and obtain the barbell diffeomorphism $\tilde{\beta}$ of $M$.
  We have
  \begin{align*}
    \Lambda(\tilde{\beta}\circ \S_d(\Sigma)) &= \Dax(\tilde{\beta} \circ \S_d, \S_d) \\
    &= \mathfrak{p} \circ \Dax_{M_2}(\beta \circ \S_d|_{H_2}, \S_d|_{H_2}) \\
    &= [g] + [g^{-1}].
  \end{align*}
  The first equation holds by the definition of $\Lambda$.
  The second equation holds because of Definition \ref{defofdax} and the fact that $\tilde{\beta}$ is supported on $M_2$.
  The third equation holds because $\beta \circ \S_d|_{H_2}$ is obtained by tubing a self-referential disk
  on $\S_d|_{H_2}$ along a path representing an element conjugate to $g$ in $\pi$.
\end{proof}

\section{The mapping class group of $\Sigma \ltimes S^2$}\label{sec:mcg}

In this section, we will use the Dax invariants to construct the surjective homomorphism $\hat{\Phi} :\MCG(M) \rightarrow \ZZ^\infty$ in Theorem \ref{thm2}.

Let $\Diff_*(M)$ be the group consisting of the self-diffeomorphisms of $M$ fixing the basepoint $b$.
Define $\MCG_*(M) = \pi_0 \Diff_*(M)$.
Let $\aut_*(M)$ be the group consisting of the pointed homotopy equivalences between $M$ and itself, up to homotopy fixing the basepoint.
There is a canonical projection homomorphism $p_0: \MCG_*(M) \rightarrow \aut_*(M)$.
Let $\Diff^+_*(M)$ be the group consisting of the pointed orientation-preserving self-diffeomorphisms of $M$.
Let $\MCG^+_*(M) = \pi_0 \Diff^+_*(M)$,
which is a subgroup of $\MCG_*(M)$ of index 2.

\begin{prop}\label{phi0}
  There is a surjective homomorphism given by
  $$ \Phi_0 : \ker p_0 \rightarrow \ZZ[\mathcal{C}]^\sigma, \quad [f] \mapsto \Dax(\S_0, f\circ\S_0). $$
\end{prop}

\begin{proof}
  We first show that $\Phi_0$ is a homomorphism. Given $[f], [g] \in \ker p_0$, it holds the following
  \begin{align*}
    \Phi_0([f \circ g]) &= \Dax(\S_0, f \circ g \circ \S_0) \\
    &= \Dax(f \circ \S_0, f \circ g \circ \S_0) + \Dax(\S_0, f \circ \S_0) \\
    &= \Dax(\S_0, g \circ \S_0) + \Dax(\S_0, f \circ \S_0) \\
    &= \Phi_0([g]) + \Phi_0([f]). 
  \end{align*}

  For any $c \in \mathcal{C}$, 
  there is a self-referential barbell diffeomorphism $\beta_{(S_0,m,\hat{\gamma})}$ whose associated loop is $c$.
  See the proof of Proposition \ref{cor4.3} for the construction of self-referential barbell diffeomorphisms.
  According to \cite[Construction 5.3]{G19} and \cite[Definition 2.1]{G21},
  the action of $\beta_{(S_0,m,\hat{\gamma})}$ on $D_0$ 
  is equivalent to connecting $m$ and $D_0$ by a tube along $\hat{\gamma}$,
  as well as attaching a self-referential disk to $D_0$ along $\hat{\gamma}$. 
  By Definition \ref{defofdax} and \cite[Theorem 4.9(ii)]{G21},
  we have 
  $$\Dax(\S_0, \beta_{(S_0,m,\hat{\gamma})} \circ \S_0) = \mathfrak{p} \circ \Dax_{M_2}(D_0,\beta_{(S_0,m,\hat{\gamma})}(D_0)) = c + c^{-1}.$$ 
  Thus, we obtain $\Phi_0([\beta]) = c + c^{-1}.$
  Since $\ZZ[\mathcal{C}]^{\sigma}$ is an abelian group generated by elements of the form $c + c^{-1}$,
  $\Phi_0$ is a surjective homomorphism.
\end{proof}

Recall that homology group $H_2(M)$ is isomorphic to $\ZZ^2$.
We can pick a basis $\{ (\S_0)_*([\Sigma]), i_*([S^2]) \}$ for $H_2(M)$,
where $[\Sigma]$ and $[S^2]$ are fundamental classes of $\Sigma$ and $S^2$ respectively, 
and $i: S^2 \hookrightarrow M$ is the inclusion of fiber at basepoint.
For convenience, we denote the basis $\{ (\S_0)_*([\Sigma]), i_*([S^2]) \}$ as $\{x,y\}$.
Under such choice of basis, the intersection form on $H_2(M)$ can be describe by the matrix
\begin{equation*}
\begin{pmatrix}
1 & 1 \\
1 & 0
\end{pmatrix}.
\end{equation*}
The following lemma describes the action of self-diffeomorphisms of $M$ on the homology group $H_2(M)$.

\begin{lem}\label{diff+}
  Let $f: M \rightarrow M$ be a self-diffeomorphism of $M$.
  \begin{itemize}
    \item If $f$ preserves the orientation, it holds that $f_*(x) = \pm x$;
    \item If $f$ reverses the orientation, it holds that $f_*(x) = \pm (x-y)$.
  \end{itemize}
\end{lem}

\begin{proof}
  Suppose that
  \begin{equation*}
    f_* 
    \begin{pmatrix}
      x \\
      y
    \end{pmatrix}
    =
    \begin{pmatrix}
      a & b \\
      c & d
    \end{pmatrix}
    \begin{pmatrix}
      x \\
      y
    \end{pmatrix},
  \end{equation*}
  for some integers $a,b,c,d$.
  Since $f$ is a self-diffeomorphism, we have $ad-bc = \pm 1$.
  Note that $c$ can be identified to the degree of the composition $S^2 \xrightarrow{i} M \xrightarrow{f} M \xrightarrow{\rm pr} \Sigma$, which can only be zero.
  Hence $ad = \pm 1$.

  If $f$ preserve the orientation, we have
  $ 1 = x \cdot x = f_*(x) \cdot f_*(x) = (ax+by) \cdot (ax+by) = 1 + 2ab. $
  It deduce that $b = 0$.
  If $f$ reverse the orientation, we have
  $ -1 = - x \cdot x = f_*(x) \cdot f_*(x) = 1+2ab.$
  It implies that $ab = -1$, hence $b = -a$.
\end{proof}

Given $[f] \in \MCG^+_*(M)$ and a pointed embedding $i : \Sigma \hookrightarrow M$ homotopic to $\S_d$ relative to the basepoint.
We need to reparameterize the embedding $f \circ i: \Sigma \hookrightarrow M$ to construct a new embedding homotopic to $i$.
The composition 
$$\Sigma \xrightarrow{i} M \xrightarrow{f} M \xrightarrow{\rm pr} \Sigma$$
induces an automorphism on $\pi$.
Remind that there is a pointed self-diffeomorphism $f_{\dagger}$ of $\Sigma$ inducing the same automorphism on $\pi$. 
Define the pointed embedding $i^f: \Sigma \hookrightarrow M$ to be the composition
$$ \Sigma \xrightarrow{f_\dagger} \Sigma \xrightarrow{i} M \xrightarrow{f} M. $$

\begin{lem}
  The pointed map $i^f$ is homotopic to $\S_d$ relative to the basepoint.
\end{lem}

\begin{proof}
  We use the identification of $[\Sigma,M]$ constructed in Proposition \ref{bijmap}.
  Due to the reparameterization $f_\dagger$, the induced automorphisms of $i_f$ and $\S_d$ on $\pi$ are identical.
  Thus the element $\phi_2([i^f],[\S_d])$ of $\pi_2(M)$ is well-defined. 
  We refer the readers to the proof of Proposition \ref{bijmap} for the concrete definition for $\phi_2$. 
  Under the isomorphism $\pi_2(M) \cong \ZZ$, 
  $\phi_2([i^f],[\S_d])$ can be identified to the difference $(i^f \cap \S_0) - (\S_d \cap \S_0)$ of the algebraic intersection number in $M$. 
  The difference is always zero because the following equations holds in $H_2(M)$ 
  $$ (i^f)_*([\Sigma]) = f_* \circ (i \circ f_\dagger)_*([\Sigma]) = x = (\S_0)_*([\Sigma]).$$ 
  The second equation holds by Lemma \ref{diff+} and the construction of $f_\dagger$.
\end{proof}

Denote the automorphism group of $\pi$ as $\aut(\pi)$. 
The group $\aut(\pi) \times \ZZ/2$ act on $\mathcal{C}$, 
where the component $\ZZ/2$ stands for taking inverse. 
Denote the set consisting of the orbits as $A$, which is infinite due to \cite[Lemma 5.10]{LWXZ25}.
There is a canonical projection $\mathscr{P} : \ZZ[\mathcal{C}] \rightarrow \ZZ[A]$.

\begin{prop}\label{Phi1}
  There is a surjective homomorphism given by
  $$\Phi_1: \MCG_*^+(M) \rightarrow \ZZ[A], \quad [f] \mapsto \frac{1}{2} \cdot \mathscr{P} \circ \Dax(\S_0,\S_0^f).$$
\end{prop}

\begin{proof}
  Since the codomain of the Dax invariant is $\ZZ[\mathcal{C}]^\sigma$, 
  dividing by two in the definition of $\Phi_1$ is well-defined.

  We pick embeddings $i_1,i_2$ homotopic to $\S_0$ and a mapping class $[f]$ in $\MCG_*^+(M)$. 
  We denote the unparameterized embedded surfaces $i_j(\Sigma)$ as $\Sigma_j$ for $j = 1,2$.
  By Lemma \ref{parameter} and Lemma \ref{naturality}, we have
  \begin{equation*}
    \begin{split}
      \Dax(i_1^f,i_2^f) &= \Dax(f(\Sigma_1),f(\Sigma_2)) \\
      &= f_* \Dax(\Sigma_1,\Sigma_2)\\
      &= f_* \Dax(i_1,i_2).
    \end{split}
  \end{equation*}
  Let $[f],[g]$ be elements in $\MCG_*^+(M)$. 
  Note that we may pick $(f \circ g)_\dagger$ to be $g_\dagger \circ f_\dagger$.
  It deduce that $\S_0^{f \circ g} = (\S_0^g)^f$.
  Then $\Phi_1$ is a homomorphism, because we have
  \begin{equation*}
    \begin{split}
      \Phi_1([f \circ g]) &= \frac{1}{2} \cdot \mathscr{P} \circ \Dax(\S_0,\S_0^{f \circ g}) \\
      &= \frac{1}{2} \cdot \mathscr{P} (\Dax(\S_0,\S_0^f) + \Dax(\S_0^f, \S_0^{f \circ g})) \\
      &= \frac{1}{2} \cdot \mathscr{P} (\Dax(\S_0,\S_0^f) + \Dax(\S_0^f, (\S_0^g)^f)) \\
      &= \frac{1}{2} \cdot \mathscr{P} (\Dax(\S_0,\S_0^f) + f_* \Dax(\S_0,\S_0^g)) \\
      &= \Phi_1([f]) + \Phi_1([g]).
    \end{split}
  \end{equation*}

  For any $a \in A$, we have $c \in \mathcal{C}$ such that $\mathscr{P}(c) = a$.
  According to Proposition \ref{phi0}, there is a class $[f] \in \ker p_0$ such that $\Dax(\S_0,f \circ \S_0) = c + c^{-1}$.
  It deduce that $\Phi_1([f]) = a$. Therefore, $\Phi_1$ is surjective.
  It follows that our results hold.
\end{proof}

Note that the map $\Diff^+(M) \rightarrow M$ given by evaluating at basepoint is a fibration,
whose homotopy fiber is $\Diff^+_*(M)$.
Thus there is a long exact sequence of homotopy groups
$$ \cdots \rightarrow \pi \xrightarrow{\partial} \MCG^+_*(M) \rightarrow \MCG^+(M) \rightarrow 0. $$

\begin{prop}
  The restriction of $\Phi_2$ on ${\rm Im}(\partial)$ is trivial.
  Hence the homomorphism $\Phi_1$ induces a surjective homomorphism
  $$\Phi_2: \MCG^+(M) \cong \MCG^+_*(M) / {\rm Im}(\partial) \rightarrow \ZZ[A]. $$
\end{prop}

\begin{proof}
  We use the model of $M$ constructed in (\ref{construct}) and (\ref{model}).
  Given $[\gamma] \in \pi$, assume that $\gamma$ is a based loop factor through $\Sigma$. 
  Namely, there is a based loop $\gamma_0$ in $\Sigma$ satisfying $\gamma = \S_0 \circ \gamma_0$.
  By a transversal argument, we may assume that $\gamma_0$ has no intersection with the disk $D$ in (\ref{construct}).
  Denote the arc-pushing map along $\gamma_0$ as $g: \Sigma \rightarrow \Sigma$, which is supported away from $D$.
  Since the mapping class $\partial [\gamma]$ is obtained by arc-pushing along $\gamma$, 
  it has representative
  $$ f := (g|_{\Sigma \setminus D} \times {\rm id}_{S^2}) \cup {\rm id}_{D \times S^2} : M \rightarrow M. $$
  The composition map $ {\rm pr} \circ f \circ \S_0 $ is identical to $g$.
  Thus we can pick $f_{\dagger} = g$.
  It deduced that $\S_0^{f} = f \circ \S_0 \circ g^{-1} = \S_0$.
  Therefore, $\Phi_1(\partial [\gamma]) = \Dax(\S_0,\S_0^f) = \Dax(\S_0,\S_0) = 0$, which concludes our proof. 
\end{proof}

By the model of $M$ constructed in (\ref{construct}) and (\ref{model}), 
we can construct an orientation-reversing self-diffeomorphism $\tau$ of $M$, given by
$$ \tau(p,(x,y,z)^T) = (p,(x,y,-z)^T). $$
Note that $\tau$ is an involution of $M$.
Namely, we have $\tau^2 = {\rm id}$.
We need the following lemma to study the action of $\tau$ on $\S_0$.

\begin{lem}\label{taus0}
  The embedding $\tau \circ \S_0$ is isotopic to $\S_{-1}$.
\end{lem}

\begin{proof}
  We may use the model of $M$ and the constructions of $\S_d$ introduced in Subsection \ref{subsec:tstspace}.
  Assume that $\S_0$ and $\S_{-1}$ are identical away from $D \subset \Sigma$.
  Let $N$ be a neighbourhood of $D$ that diffeomorphic to the 2-disk.
  Fix a smooth function $\rho: \Sigma \rightarrow [0,1]$ satisfying $\rho(\Sigma \setminus N) = 1$ and $\rho(D) = 0$.
  Construct an isotopy
  $$ H: \Sigma \times [0,1] \rightarrow M, \quad (p,t) \mapsto (p,(0, \sin\pi\rho(p)t,-\cos\pi\rho(p)t)^T).$$
  The isotopy $H$ is well-defined because $\rho(D) = 0$.
  It holds that $H(p,0) = \tau \circ \S_0(p)$ and $H(q,1) = \S_{-1}(q)$ for all $q \in \Sigma \setminus N$.
  We denote the embedding $H(-,1)$ as $\S'$.
  Restricting the $S^2$-bundle $p:M \rightarrow \Sigma$ on $N$, we obtain the trivial $S^2$-bundle $p^{-1}(N) \cong N \times S^2$.  
  It is sufficient to show that two such neat embeddings $\S'|_{N}, \S_{-1}|_{N}: D^2 \hookrightarrow p^{-1}(N)$ are isotopic relative to $\partial D^2$.
  
  By \cite[Theorem 10.4]{G20}, we only need to show that $\S'|_{N}, \S_{-1}|_{N}$ are homologous in $H_2(p^{-1}(N),\S'(\partial N))$.
  Namely, we need to show that$(\S'|_{N})_*(g) = (\S_{-1}|_{N})_*(g)$, where $g$ is a generator of $H_2(N,\partial N)$.
  By the Mayer-Vietoris sequence, there is a split short exact sequence 
  $$0 \rightarrow H_2(p^{-1}(N)) \xrightarrow{p_*} H_2(p^{-1}(N),\S'(\partial N)) \xrightarrow{\partial_*} H_1(\S'(\partial N)) \rightarrow 0.$$
  Since $\S'|_{\partial N} = \S_{-1}|_{\partial N}$, it holds that $\partial_*(\S'|_{N})_*(g) = \partial_*(\S_{-1}|_{N})_*(g) \in H_1(\S'(\partial N))$.
  Meanwhile, the coordinate with respect to the dircet summand $H_2(p^{-1}(N)) \cong \ZZ$ can be detected through the following homomorphism
  $$ {\rm pr}_* : H_2(p^{-1}(N),\S'(\partial N)) \rightarrow H_2(S^2,{\rm pt}) \cong \ZZ, $$
  where ${\rm pr}:p^{-1}(N) \xrightarrow{\cong} N \times S^2 \rightarrow S^2$ is a projection map.
  We fix an diffeomorphism
  \begin{equation}\label{diffp-1N}
    p^{-1}(N) \rightarrow N \times S^2, \quad 
    \begin{cases}
      (r,\theta,v) \mapsto (r,\theta,v), &(r,\theta) \in D \\
      (r,\theta,v) \mapsto (r,\theta,\phi(\theta)^{-1}v), &(r,\theta) \in N \setminus D.
    \end{cases}
  \end{equation}
  See (\ref{model}) for the expression of $\phi$.
  Then take 
  $$ N \times S^2 \rightarrow S^2, \quad (x,v) \mapsto v.$$
  Let ${\rm pr}:p^{-1}(N) \rightarrow S^2$ be the composition of the two maps above.
  The map ${\rm pr} \circ \S'|_{N} : N \rightarrow S^2$ maps $\partial N$ to $(0,0,1)^T$.
  It induces a map $\eta': S^2 = N/\partial N \rightarrow S^2$.
  Under such isomorphism $H_2(S^2,{\rm pt}) \cong \ZZ$, we have
  $$ {\rm pr}_* \circ (\S'|_{N})_*(g) = \deg (\eta') \in \ZZ.$$
  Up to a homotopy on $N$, the explicit expression of $\eta'$ can be given as
  \begin{equation*}
    \eta: N/\partial N \rightarrow S^2, \quad (r,\theta) \mapsto \phi(\theta)^{-1}(0,\sin \pi r, \cos \pi r)^T, \quad r \in [0,1], \theta \in S^1.
  \end{equation*}
  It shows that $\deg (\eta') = -1$. 
  By the same argument and the construction of $\S_{-1}$, 
  we have ${\rm pr}_* \circ (\S_{-1}|_{N})_*(g) = -1$.
  Therefore, we have $\S'|_{N}, \S_{-1}|_{N}$ are homologous in $H_2(p^{-1}(N),\S'(\partial N))$, which concludes our proof.
\end{proof}

For any $[f] \in \MCG^+(M)$, we define that 
$$ \Phi_2'([f]) = \Phi_2([\tau \circ f \circ \tau]).$$
Then $\Phi_2'$ is also a surjective homomorphism from $\MCG^+(M)$ to $\ZZ[A]$.
The group $\MCG(M)$ can be decomposed into a disjoint union of a subgroup $\MCG^+(M)$ and a coset $[\tau] \cdot \MCG^+(M)$.
By this decomposition, we can finally define the desired homomorphism.

\begin{thm}
  There is a homomorphism given by
  \begin{equation}
    \Phi: \MCG(M) \rightarrow \ZZ[A], \quad
    \begin{cases}
      [f] \mapsto \Phi_2([f]) + \Phi_2'([f]), & [f] \in \MCG^+(M), \\
      [f] \mapsto \Phi_2([\tau \circ f]) + \Phi_2'([\tau \circ f]), &[f] \in [\tau] \cdot \MCG^+(M).
    \end{cases}
  \end{equation}
  The image of $\Phi$ is of countably infinite rank.
\end{thm}

\begin{proof}
  For any $[f] \in \MCG^+(M), [g] \in [\tau] \cdot \MCG^+(M)$, we have
  \begin{equation*}
    \begin{split}
      \Phi([f \circ g]) &= \Phi_2([\tau \circ f \circ g]) + \Phi_2'([\tau \circ f \circ g]) \\
      &= \Phi_2'([f]) + \Phi_2([\tau \circ g]) + \Phi_2([f \circ g \circ \tau]) \\
      &= \Phi_2'([f]) + \Phi_2([\tau \circ g]) + \Phi_2([f]) + \Phi_2'([\tau \circ g])\\
      &= \Phi([f]) + \Phi([g]),
    \end{split}
  \end{equation*}
  and
  \begin{equation*}
    \begin{split}
      \Phi([g \circ f]) &= \Phi_2([\tau \circ g \circ f]) + \Phi_2'([\tau \circ g \circ f]) \\
      &= \Phi_2([\tau \circ g]) + \Phi_2([f]) + \Phi_2([g \circ f \circ \tau]) \\
      &= \Phi_2([\tau \circ g]) + \Phi_2([f]) + \Phi_2'([\tau \circ g]) + \Phi_2'([f])\\
      &= \Phi([g]) + \Phi([f]),
    \end{split}
  \end{equation*}
  For any $[f], [g] \in [\tau] \cdot \MCG^+(M)$, we have
  \begin{equation*}
    \begin{split}
      \Phi([f \circ g]) &= \Phi_2([f \circ g]) + \Phi_2'([f \circ g]) \\
      &= \Phi_2([f \circ \tau]) + \Phi_2([\tau \circ g]) + \Phi_2([\tau \circ f \circ g \circ \tau]) \\
      &= \Phi_2'([\tau \circ f]) + \Phi_2([\tau \circ g]) + \Phi_2([\tau \circ f]) + \Phi_2'([\tau \circ g])\\
      &= \Phi([f]) + \Phi([g]).
    \end{split}
  \end{equation*}
  Thus, we have shown that $\Phi$ is a homomorphism.

  It remains to study the image of $\Phi$.
  Given $[f] \in \MCG^+(M)$,
  we fix the self-diffeomorphism $f_{\dagger}$ of $\Sigma$.
  Then $\S_0^f = f \circ \S_0 \circ f_{\dagger}^{-1}$.
  Note that the projection ${\rm pr}: M \rightarrow \Sigma$ is invariant under the action of $\tau$, 
  and $({\rm pr})_* : \pi_1(M) \rightarrow \pi_1(\Sigma)$ is an isomorphism.
  Thus, the homomorphism $\tau_*$ on $\pi_1(M)$ induced by $\tau$ is the identity map.
  It deduces that $(\tau \circ f \circ \tau)_{\dagger} = f_\dagger$.
  We have
  \begin{equation*}
    \begin{split}
      \Phi_2'([f]) &= \Dax(\S_0,\S_0^{\tau f \tau}) \\
      &= \Dax(\S_0, \tau \circ f \circ \tau \circ \S_0 \circ f_\dagger^{-1} ) \\
      &= \Dax(\tau \circ \S_0, f \circ \tau \circ \S_0 \circ f_\dagger^{-1} ) \\
      &= \Dax(\S_{-1}, f \circ \S_{-1} \circ f_\dagger^{-1})\\
      &= \Dax(\S_{-1}, \S_{-1}^f).
    \end{split}
  \end{equation*}
  The second equation holds because $(\tau \circ f \circ \tau)_{\dagger} = f_\dagger$.
  The third equation holds because of Lemma \ref{naturality} and the facts that $\tau^2 = {\rm id}$ and $\tau_* = {\rm id}_{\pi}$.
  The fourth equation holds because of Lemma \ref{taus0}.

  For any $a \in A$, 
  there is a self-referential barbell diffeomorphism $\beta$ whose associated loop $c$ satisfies $\mathscr{P}(c) = a$.
  Without loss of generality, 
  we may assume $\S_0$ and $\S_{-1}$ are identical away from the disk $D$ chosen in the model (\ref{construct}). 
  At the same time, we may also assume that $\beta$ is supported away from the restricted trivial bundle $D \times S^2 \subset M$ and the basepoint. 
  Note that $[\beta] \in \ker p_0$. Therefore, it holds that
  \begin{equation*}
    \begin{split}
      \Phi([\beta]) &= \Phi_2([\beta]) + \Phi_2'([\beta]) \\
      &= \frac{1}{2} \left(\mathscr{P} \circ \Dax(\S_0,\beta \circ \S_0) + \mathscr{P} \circ \Dax(\S_{-1},\beta \circ \S_{-1})\right) \\
      &= \mathscr{P}(c+c^{-1}) \\
      &= 2a.
    \end{split}
  \end{equation*}
  It deduces that the image of $\Phi$ contains the subgroup $2\ZZ[A]$.
  Thus, as a free abelian group, the rank of ${\rm Im}(\Phi)$ equals to the cardinality of $A$. 
\end{proof}

By restricting the codomain of $\Phi$ to its image, we obtain a surjective homomorphism
$ \hat{\Phi} : \MCG(M) \rightarrow \ZZ^{\infty}. $
However, it remains such a question.

\begin{Q}
  Is the image of $\Phi$ strictly larger than $2\ZZ[A]$ ?
\end{Q}

\bibliographystyle{plain}
\bibliography{daxinv}

\end{document}